\documentclass[letterpaper,11pt,reqno]{amsart}

\makeatletter
\usepackage[T1]{fontenc}
\usepackage[utf8]{inputenc}
\usepackage{lmodern}
\usepackage{amssymb}
\usepackage{latexsym}
\usepackage{amsbsy}
\usepackage{amsfonts}
\usepackage{hyperref}
\usepackage{graphicx}
\usepackage{enumerate}
\usepackage{color}
\usepackage{cite}

\definecolor{darkblue}{rgb}{0,0,0.5}
\definecolor{darkred}{rgb}{0.5,0,0}
\definecolor{darkgreen}{rgb}{0,0.5,0}

\hypersetup{
  unicode=true,          
  pdffitwindow=false,     
  pdfstartview={FitH},    
  pdftitle={},    
  pdfauthor={},     
  pdfsubject={},   
  pdfkeywords={}, 
  colorlinks=true,       
  linkcolor=darkblue,          
  citecolor=darkred,        
  filecolor=magenta,      
  urlcolor=darkgreen           
}

\def\marginpar#1{\ignorespaces}

\def\Var{\textrm{Var}}

\def\to{\rightarrow}

\def\P{{\mathbb P}}        
\def\E{{\mathbb E}}        
\def\te{\rightarrow}

\def\Hn{H(n)}
\def\SH{H^*}

\def\Li{\textrm{Li}}

\renewcommand{\P}{\mathbb{P}}

\newcommand{\thrftwo}[6]{ \,_3F_2 \left( \left.  {\genfrac{}{}{0pt}{}{#1, #2, #3 }{#4, #5}} \right| #6 \right) }

\DeclareMathOperator\re{Re}

\newcommand{\eqth}{ \stackrel{\theta}{=} }
\newcommand{\eqone}{ \stackrel{1}{=} }
\newcommand{\ed}{ \stackrel{(d)}{=} }
\renewcommand{\P}{{\mathbb P }}

\newtheorem{theorem}{Theorem}[section]
\newtheorem{lemma}[theorem]{Lemma}
\newtheorem{proposition}[theorem]{Proposition}
\newtheorem{corollary}[theorem]{Corollary}

\newtheorem{conj}[theorem]{Conjecture}
\numberwithin{equation}{section}
\makeatother
\begin{document}
\title[Renewal sequences and record chains related to multiple zeta sums]{Renewal sequences and record chains related to multiple zeta sums}

\author[Jean-Jil Duchamps]{{Jean-Jil Duchamps}}
\address{Université Pierre et Marie Curie -- UPMC Univ Paris 6 (LPMA).
} \email{jean-jil.duchamps@normalesup.org}

\author[Jim Pitman]{{Jim} Pitman}
\address{Statistics department, University of California, Berkeley.
} \email{pitman@stat.berkeley.edu}

\author[Wenpin Tang]{{Wenpin} Tang}
\address{Statistics department, University of California, Berkeley.
} \email{wenpintang@stat.berkeley.edu}

\date{\today} 

%

\begin{abstract}
For the random interval partition of $[0,1]$ generated by the uniform stick-breaking scheme 
known as GEM$(1)$, 
let $u_k$ be the probability that the first $k$ intervals created by the stick-breaking scheme are also the first $k$ intervals to be discovered 
in a process of uniform random sampling of points from $[0,1]$.  Then $u_k$ is a renewal sequence. We prove that
$u_k$ is a rational linear combination of the real numbers $1, \zeta(2), \ldots, \zeta(k)$ where $\zeta$ is the Riemann zeta function,
and show that $u_k$ has limit $1/3$ as $k \to \infty$. Related results provide probabilistic interpretations of some
multiple zeta values in terms of a Markov chain derived from the interval partition. This Markov chain has the structure of a weak record chain.
Similar results are given for the GEM$(\theta)$ model, with beta$(1,\theta)$ instead of uniform stick-breaking factors, and
for another more algebraic derivation of renewal sequences from the Riemann zeta function.
\end{abstract}

\maketitle


\setcounter{tocdepth}{1}
\tableofcontents
\section{Introduction}
\quad Consider the following two sequences of random subsets of the set $[n]:= \{1, \ldots, n \}$,
generated by listing the cycles of a uniform random permutation $\pi_n$ of $[n]$ in two different orders: 
for a permutation $\pi_n$ with $K_n$ cycles, 
\begin{itemize} 
\item  let $C_{1:n}, C_{2:n}, \ldots$  be the cycles of $\pi_n$ in {\em order of least elements},
so $C_{1:n}$ is the cycle of $\pi_n$ containing $1$, if $C_{1:n} \ne [n]$ then $C_{2:n}$ is the cycle of $\pi_n$
containing the least $j \in [n]$ with $j \notin C_{1:n}$, and so on, 
with $C_{k:n} = \emptyset$ if $k > K_n$;
\item let $R_{1:n}, R_{2:n}, \ldots, R_{K_n:n}$ be the same cycles in  {\em order of greatest elements}, so
$R_{1:n}$ is the cycle of $\pi_n$ containing $n$, if $R_{1:n} \ne [n]$ then $R_{2:n}$ is the cycle of $\pi_n$
containing the greatest $j \in [n]$ with $j \notin R_{1:n}$, and so on, with $R_{k:n} = \emptyset$ if $k > K_n$.
\end{itemize}
For $1 \le k \le n$ define the probability 
\begin{equation}
\label{eq:ukn}
u_{k:n} := \P \left( \cup_{i=1}^k C_{i:n} = \cup _{i = 1}^k R_{i:n} \right),
\end{equation}
the probability that the same collection of $k$ cycles appears as the first $k$ cycles in  both orders.
It is elementary that the number of elements of $C_{1:n}$ has a discrete uniform distribution on $[n]$,
and the same is true for $R_{1:n}$.
Since $u_{1:n}$ is the probability that both $1$ and $n$ fall in the same cycle of $\pi_n$, it follows easily that $u_{1:n} = 1/2$ for every $n \ge 2$.
For $k \ge 2$ it is easy to give multiple summation formulas for $u_{k:n}$. Such formulas show 
that $u_{k:n}$  has some dependence on $n$ for $k \ge 2$, with limits
\begin{equation}
\label{uklim}
\lim_{n\to \infty } u_{k:n} = u_k \mbox{ for each } k = 1,2, \ldots, 
\end{equation}
which may be described as follows.
It is known \cite[p. 25]{abtlogcs} that the asymptotic structure of sizes of cycles of $\pi_n$, when listed in either order, and normalized by $n$, is that of the
sequence of lengths of subintervals of $[0,1]$ defined by the {\em uniform stick-breaking scheme}:
\begin{equation}
\label{wk:intro}
P_1:= W_1, \qquad P_2:= (1-W_1) W_2, \qquad P_3:= (1-W_1) ( 1 - W_2) W_3 , \cdots
\end{equation}
where the $W_i$ are i.i.d. uniform $[0,1]$ variables.  In more detail, 
\begin{equation}
P_k:= |I_k| \mbox{ where } I_k := [R_{k-1}, R_k) \mbox{ with } R_k = \sum_{i=1}^k P_k = 1 - \prod_{i=1}^k ( 1 - W_i ),
\end{equation}
with the convention $R_0=0$.
The distribution of lengths $(P_1,P_2, \ldots)$
so obtained, with $\sum_{i} P_i = 1$ almost surely,  is known as the GEM$(1)$ model, after Griffiths, Engen and McCloskey.
The limit probabilities $u_k$ are easily evaluated directly in terms of the limit model, as follows.
Let $(U_1, U_2, \ldots)$ be an i.i.d.\ uniform $[0,1]$ sequence of {\em sample points} independent of $(P_1, P_2, \ldots)$.
Say that an interval $I_k$ has been discovered by time $i$ if it contains at least one of the sample points $U_1, U_2, \ldots, U_i$.
The sampling process imposes a new order of discovery on the intervals $I_k$, which describes the large $n$ limit structure of the random
permutation of cycles of $\pi_n$ described above. 
The limit $u_k$ in \eqref{uklim} is $u_k:= \P(E_k)$,
the probability of the event $E_k$ in the limit model that the union of the first $k$ intervals to be discovered in the sampling process equals
$\cup_{i=1}^k I_i = [0,R_k)$ for $R_k = P_1 + \cdots + P_k$ as above. There are several different ways to express this event $E_k$.
The most convenient for present purposes is to consider the stopping time $n(k,1)$ when the sampling process first discovers a point not in $[0,R_k)$.
If at that time $n = n(k,1)$, for $1 \le i \le k$ there is at least one sample point $U_j \in I_i$ with $1 \le j < n$, then the event $E_k$ has occurred,
and otherwise not. Thus for $k = 1,2, \ldots$
\begin{equation}\label{eq:ukintro}
u_k = \P\left( \cap_{i=1}^k ( U_j \in I_i \mbox{ for some } 1 \le j < n(k,1) ) \right)
\end{equation}
and we adopt the convention that $u_0:=1$.
This sequence $(u_k, k\geq 0)$ is a renewal sequence appearing in the study of \emph{regenerative permutations} in \cite{PT17}.
In that context it is easily shown that the limit $u_\infty := \lim_{k\to \infty} u_k$ exists, but difficult to evaluate the $u_k$ for general $k$.
However, computation of $u_k$ for the first few $k=1,2,3,\ldots$  by symbolic integration suggested 
a general formula for 
$u_k$ as a rational linear combination of the real numbers $1, \zeta(2), \ldots, \zeta(k)$ where $\zeta$ is the Riemann zeta function,
\begin{equation*}
\zeta(s): = \sum_{n = 1}^{\infty} \frac{1}{n^s} \qquad \mbox{for } \re(s) > 1.
\end{equation*}

\quad This article establishes the following result, whose proof leads to some probabilistic interpretations of multiple zeta values and harmonic sums.

\begin{proposition} 
\label{prop:conju}
The renewal sequence $(u_k,~k \ge 0)$ defined above by \eqref{eq:ukintro}
in terms of uniform stick-breaking is characterized by any one of the following equivalent conditions:
\begin{enumerate}[(i).]
\item
The sequence $(u_k,~ k \ge 0)$ is defined recursively by
\begin{equation}
\label{rec}
2 u_{k} + 3 u_{k-1} + u_{k-2} = 2 \zeta(k) \quad \mbox{with } u_0 = 1, \, u_1 = 1/2.
\end{equation}
\item
For all $k \ge 0$,
\begin{equation}
\label{rzs}
u_{k} = (-1)^{k-1} \left(2 - \frac{3}{2^k} \right) + \sum_{j=2}^{k}  (-1)^{k-j} \left(2 -\frac{1}{2^{k-j}} \right) \zeta(j).
\end{equation}
\item
For all $k \ge 0$,
\begin{equation}
\label{positive}
u_{k} = \sum_{j = 1}^{\infty} \frac{2}{j^k(j+1)(j+2)}.
\end{equation}
\item
The generating function of $(u_{k},~ k \ge 0)$ is 
\begin{equation}
\label{Uz}
U(z) : = \sum_{k=0}^{\infty}u_k z^k = \frac{2}{(1+z)(2+z)} \Bigg[ 1 +   \Bigg(2 - \gamma - \psi(1-z)\Bigg) z\Bigg],
\end{equation}
for $|z| < 1$, where 
$\gamma: = \lim_{n \rightarrow \infty} (\sum_{k=1}^n 1/k - \ln n) \approx 0.577$ is the Euler constant, and
$\psi(z): = \Gamma'(z)/\Gamma(z)$ with
$\Gamma(z): = \int_0^\infty t^{z-1} e^{-t} dt$, the digamma function. 
\end{enumerate}
\end{proposition}
The proof is given in Section \ref{sec:prop_proof}.
We are also interested in $(u_k)$ when the partition $(I_k)$ follows a more general stick-breaking scheme, where the $(W_k)$ in 
\eqref{wk:intro}
are i.i.d.\ with an arbitrary distribution on $(0,1)$.
We will develop in particular the case, for $\theta>0$, where $W_k$ follows a beta$(1,\theta)$ distribution  -- with density $\theta(1-x)^{\theta-1}$ on $[0,1]$.
The sequence $u_k$ so defined is the limit \eqref{uklim} if the distribution of the random permutation $\pi_n$ is changed from the uniform distribution on permutations of $[n]$
to the {\em Ewens $(\theta)$ distribution} on permutations of $[n]$, in which the probability of any particular permutation of $[n]$ with $k$ cycles is $\theta^k/(\theta)_n$ instead of $1/(1)_n$,
where $(\theta)_n:= \theta ( \theta + 1) \cdots ( \theta + n-1)$ is a rising factorial.
In that case the limit distribution of interval lengths $(P_1, P_2, \ldots)$ is known as the GEM$(\theta)$ distribution \cite[\S 5.4]{abtlogcs}.
Our expressions for $u_k$ in this case are less explicit.
In the following, the notation $\stackrel{\theta}{=}$ indicates evaluations for the GEM$(\theta)$ model.
For instance, it was proved in \cite[(7.16)]{PT17} that
\begin{equation}
\label{PT17formula}
u_{\infty} \stackrel{\theta}{=} \frac{\Gamma(\theta+2) \Gamma(\theta+1)}{\Gamma(2\theta+2)} \stackrel{1}{=} \frac{1}{3}.
\end{equation}

A consequence of Proposition \ref{prop:conju} is that for each $k\geq 1$, the right-hand side of \eqref{rzs} is positive -- in fact strictly greater than $u_{\infty} \stackrel{1}{=} 1/3$.
Also, the probability $f_k$ of  a first renewal at time $k$, which is determined by $u_1, \ldots, u_k$ by a well known recursion recalled later in \eqref{ufrec}, is also strictly positive.
These inequalities seem not at all obvious without the probabilistic interpretations offered here. The inequalities are reminiscent of Li's criterion \cite{Li} for the Riemann hypothesis,
 which has some probabilistic interpretations indicated in \cite[Section 2.3]{BPY}.
 The GEM$(1)$ model also arises from the asymptotics of prime factorizations \cite{DG}, but the results for sampling from GEM(1) described here do not seem easy to interpret in that setting.

\quad The interpretation of $u_k$ sketched above and detailed in \cite{PT17},
 that $u_k$ is the probability that the random order of discovery of intervals maps $[k]$ to $[k]$,
yields the following corollary.
\begin{corollary}
For ${\bf w}: = (w_1, w_2, \ldots) \in (0,1)^{\mathbb{N}_{+}}$, let $p_i({\bf w}) : = (1-w_1) \cdots (1-w_{i-1}) w_i$. Then for each $k \ge 1$, the expression
\begin{equation}
\label{identity}
\sum_{\pi \in \mathfrak{S}_k} \int_{(0,1)^k} p_{\pi(1)}({\bf w}) \prod_{i=2}^{k} \frac{p_{\pi(i)}({\bf w})}{1-\sum_{j=1}^{i-1} p_{\pi(j)}({\bf w})} dw_1 \cdots dw_k
\end{equation}
is equal to \eqref{rzs} and to \eqref{positive}, where $\mathfrak{S}_k$ is the set of permutations of the finite set $\{1,\cdots,k\}$.
\end{corollary}
The expression \eqref{rzs} gives a {\em rational zeta series expansion} of the multiple integral \eqref{identity}. 
Similar expansions also appeared in Beukers' proof \cite{Beukers} of the irrationality of $\zeta(3)$.
The expression \eqref{identity} is a sum of $k!$ positive terms, while \eqref{rzs} is a linear combination of $1,\zeta(2),\cdots,\zeta(k)$ with alternating signs. 
By symbolic integration, we can identify each term of the sum in \eqref{identity} for $k=2,3$, but some terms become difficult to evaluate for $k \ge 4$, and we have
no general formula for these terms, no direct algebraic explanation of why the terms in \eqref{identity} should sum to a rational zeta series.

\quad The Riemann zeta function plays an important role in analytic number theory \cite{Edwards,Bombieri}, and has applications in geometry \cite{Witten,Ev} and mathematical physics \cite{Berry,Kirsten}.
Connections between the Riemann zeta function and probability theory have also been explored, for example:
\begin{itemize}
\item
For each $s > 1$, the normalized terms of the Riemann zeta series define a discrete probability distribution of a random variable $Z_s$ with values on $\{1, 2, \cdots\}$, such that $\log Z_s$ 
has a compound Poisson distribution \cite{ABR,LH,Gut}.
\item
The values $\zeta(2)$ and $\zeta(3)$ emerge in the limit of large random objects \cite{Frieze,AldouS}.
\item
The values $1/\zeta(n)$ for $n = 2,3 \ldots$ arise from the limit proportion of $n$-free numbers; that is, numbers not divisible by any $n$-th power of a natural number, see \cite{EL, AN}.
\item 
The values $\zeta(1/2-n)$ for $n\geq 0$ appear in the expected first ladder height of Gaussian random walks \cite{chang_ladder_1997}.
\item
The Riemann zeta function appears in the Mellin transforms of functionals of Brownian motion and Bessel processes \cite{Williams,BPY}.
\item Conjectured bounds for the zeta function on the critical line $\Re(s) = 1/2$ can be related to branching random walks \cite{arguin_maxima_2017}.
\item
There are striking parallels between the behavior of zeros of the Riemann zeta function on the line $\Re(s) = 1/2$ and the structure of eigenvalues in random matrix theory \cite{Montgomery,KSarnak,Odly}.
\end{itemize}
In the early $1990$s, Hoffman \cite{Hoffman} and Zagier \cite{Zagier} introduced the {\em multiple zeta value}
\begin{equation}
\label{multzeta}
\zeta(s_1,\cdots,s_k): = \sum_{0< n_1 < \cdots < n_k } \frac{1}{n_1^{s_1} \cdots n_k^{s_k}},
\end{equation}
and the {\em multiple zeta-star value}
\begin{equation}
\label{multzetastar}
\zeta^{*}(s_1,\cdots,s_k): = \sum_{0< n_1 \leq \cdots \leq n_k } \frac{1}{n_1^{s_1} \cdots n_k^{s_k}},
\end{equation}
for each $k>0$, and $s_i \in \mathbb{N}_{+}: = \{1,2,\cdots\}$ with $s_1 > 1$ to ensure the convergence.
Note that the multiple zeta-star value \eqref{multzetastar} can be written as the sum of multiple zeta values:
\begin{equation*}
    \zeta^*(s_1,\cdots,s_k) = \sum_{{\bf s}^*} \zeta({\bf s}^*),
\end{equation*}
where the sum is over all ${\bf s}^* = (s_1 \square \cdots \square s_k)$, with each $\square$ filled by either a comma or a plus.
To illustrate,
\begin{align*}
    & \zeta^*(s_1,s_2) = \zeta(s_1,s_2) + \zeta(s_1+s_2), \\
    & \zeta^*(s_1,s_2,s_3) = \zeta(s_1,s_2,s_3)  + \zeta(s_1,s_2+s_3) + \zeta(s_1+s_2,s_3) + \zeta(s_1+s_2+ s_3).
\end{align*}  
See \cite{BBBL,Hoffman05,AKO} for the algebraic structure, and some evaluations of multiple zeta values.
It was proved in \cite{AET,Zhao} that the multiple zeta functions \eqref{multzeta}-\eqref{multzetastar} can also be continued meromorphically on the whole space $\mathbb{C}^k$. 

\quad These multiple zeta values appear in various contexts including algebraic geometry, knot theory, and quantum field theory, see \cite{GF17}. 
But we are not aware of any previous probabilistic interpretation of these numbers.
In this article we show how the zeta values $\zeta(2),\zeta(3),\ldots$ and \eqref{multzeta}-\eqref{multzetastar} arise in the renewal sequence $(u_k)$ associated with the discovery of intervals for a GEM$(1)$ partition of $[0,1]$.
Equivalently the same sequence $(u_k)$ can be expressed in terms of a GEM$(1)$-biased permutation of $\mathbb{N}_{+}$ \cite{PT17}, or of the {\em Bernoulli sieve} \cite{Gnedinsieve} driven by the GEM$(1)$ distribution.

\bigskip

{\bf Organization of the paper:} The rest of the paper is organized as follows.
\begin{itemize}
\item In Section \ref{sec:gem_theta_records} we introduce the main tool of our analysis, a Markov chain $(\widehat{Q}_k)$
derived from
the discovery process of subintervals 
in the GEM$(\theta)$ stick-breaking model, and show its equality in distribution with a weak record chain.
\item In Section \ref{sec:prop_proof} we give the proof of Proposition \ref{prop:conju} for $\theta = 1$, and provide some partial results for general $\theta$.
\item In Section \ref{sec:renewal_seq} we define a number of renewal sequences satisfying a recursion involving the Riemann zeta function.
\item In Section \ref{sec:gem_one} we specialize again to $\theta=1$ and examine further the distribution of the Markov chain $(\widehat{Q}_k)$, deriving expressions involving iterated harmonic sums and zeta values.
\item In Section \ref{sec:u2} we derive a formula for $u_{2:n}$ associated with random permutations, which provides evaluation of $u_2$ for general $\theta$ as the limit.
Among those we identify the sequence $(u_k)$ defined by \eqref{eq:ukintro} in the GEM$(1)$ model.
\end{itemize}
\section{One-parameter Markov chains and record processes}
\label{sec:gem_theta_records}

\quad Recall the definition \eqref{wk:intro} of the length $P_k = |I_k|$ of the $k$-th interval in a stick-breaking partition and the uniform sequence $(U_i)$ of points that we use to discover intervals.
Now define a random sequence of positive integers $(X_i)$ by setting
\begin{equation}\label{eq:def_xi}
X_i := k \iff U_i \in I_k .
\end{equation} 
So $X_i$ is the rank of the interval in which the $i$-th sample point $U_i$ falls.
Conditionally given the sequence of interval lengths $(P_1,P_2, \ldots)$, the
$X_i$ are i.i.d. according to this distribution on $\mathbb{N}_{+} := \{1,2, \ldots \}$.
Formula \eqref{eq:ukintro} can be recast as
\begin{equation}\label{eq:uk_redef}
u_k = \P\left (\{X_1, X_2, \ldots, X_{n(k,1)-1}\}=\{1,2,\ldots,k\}\right ),
\end{equation}
where $n(k,1) = \inf\{i \geq 1, \, X_i \geq k+1\}$.

\quad The key to our analysis is the Markov chain $(\widehat{Q}_k)$ given by the following lemma from \cite[Lemma 7.1]{PT17}.
This lemma is suggested by work of Gnedin and coauthors on the Bernoulli sieve \cite{Gsmall,GIM}, and subsequent work on extremes and gaps in sampling from a RAM by Pitman and Yakubovich \cite{PY17, P17}.
\begin{lemma}
\label{lemma:PT}
Let $X_1, X_2, \ldots$ be as in \eqref{eq:def_xi} for a stick-breaking partition with i.i.d.\ factors $W_i \stackrel{(d)}{=} W$ as in \eqref{wk:intro}
for some distribution of $W$ on $(0,1)$.
For $n \in \mathbb{N}_{+}$ and $k = 0,1, \ldots$ let 
\begin{equation}
Q_n^*(k):= 
\sum_{i=1}^n 1 (X_i > k ) = \sum_{i=1}^n 1 (U_i \ge  R_k )
\end{equation}
represent the number of the first $n$ sample points which land outside the union $[0, R_k)$ of the first $k$ intervals. For $m = 1,2, \ldots$ 
let $n(k,m):= \min \{n : Q_n^*(k) = m \}$ be the first time $n$ that there are $m$ sample points outside the first $k$ intervals. 
Then:
\begin{enumerate}[(i).]
\item
For each $k$ and $m$ there is the equality of joint distributions
\begin{equation}
\left( Q_{n(k,m)} ^* (k-j), 0 \le j \le k \right) \stackrel{(d)}{=} \left( \widehat{Q}_j, 0 \le j \le k \,\middle|\, \widehat{Q}_0 = m \right)
\end{equation}
where $(\widehat{Q}_0, \widehat{Q}_1, \ldots)$ with $1 \le \widehat{Q}_0 \le \widehat{Q}_1 \cdots$ is a Markov chain with state space 
$\mathbb{N}_{+}$ 
and stationary transition probability function
\begin{equation}
\label{hatqdef}
\widehat{q}(m,n) :=  \binom{n-1}{m-1}\mathbb{E} W^{n-m} (1-W)^m \quad \mbox{for}~m \le n.
\end{equation}
\item
For each $k \ge 1$ the renewal probability $u_k$ defined by \eqref{eq:uk_redef} is given by
\begin{equation}
\label{ukform}
u_k = \mathbb{P}( \widehat{Q}_0 < \widehat{Q}_1 < \cdots < \widehat{Q}_k \,|\, \widehat{Q}_0 = 1 ).
\end{equation}
\item
The sequence $u_k$ is strictly decreasing, with limit $u_\infty \ge 0$ which is given by
\begin{equation}
\label{uinfform}
u_\infty = \mathbb{P}( \widehat{Q}_0 < \widehat{Q}_1 < \cdots \,|\, \widehat{Q}_0 = 1 ).
\end{equation}
\end{enumerate}
\end{lemma}

\quad Here we study the Markov chain $(\widehat{Q}_k)$ for the GEM$(\theta)$ partition and show that it has an interpretation as a \emph{weak record chain}.
Let $X_1,X_2, \ldots$ be a random sample from the GEM$(\theta)$ model with i.i.d.\ stick-breaking factors $W_i \stackrel{(d)}{=} W$ for $W$ following a beta$(1,\theta)$ distribution.
Consider
\begin{equation}
\label{C1}
    C^{\ell, \theta}_{k}: = \sum_{j=1}^k 1\{Q^{*}_{n(k,\ell)}(j) = Q^{*}_{n(k,\ell)}(j-1)\} \quad \mbox{for } k \ge 1,
\end{equation}
the number of empty intervals among the first $k$ intervals at the first time $n(k,\ell)$ there are $\ell$ points outside the first $k$ intervals. 
To study the random variables $C^{\ell, \theta}_{k}$, we introduce a family of one-parameter Markov chains $(\widehat{Q}_j^{\ell,\theta}, ~j \ge 0)$ with
\begin{itemize}
\item
the initial value $\widehat{Q}_0^{\ell, \theta} = \ell \in \mathbb{N}_{+}$, 
\item
the transition probability function $\widehat{q}^{\theta}(m,n)$ given by \eqref{hatqdef} for $W$ the beta$(1,\theta)$ distribution.
\end{itemize}
For $W$ the beta$(1,\theta)$ distribution, 
\begin{equation*}
    \mathbb{E}W^{n-m}(1-W)^m =  \frac{(1)_{n-m} \, \theta }{(\theta+m)_{n-m-1}} \quad \mbox{for } m \le n,
\end{equation*}
where
\begin{equation*}
    (x)_j: = x(x+1) \cdots (x+j-1) = \frac{\Gamma(x+j)}{\Gamma(x)}.
\end{equation*}
So the transition probability $\widehat{q}^{\theta}$ of the $\widehat{Q}^{\ell,\theta}$ chain is given by 
\begin{equation}
\label{transitheta}
  \widehat{q}^{\theta}(m,n) = \frac{(m)_{n-m} \, \theta }{(\theta+m)_{n-m+1}} \quad \mbox{for } m \le n.
\end{equation}
Let 
\begin{equation}
\label{occup}
G^{\ell,\theta}_{i}(k): = \sum_{j = 1}^k 1\{\widehat{Q}^{\ell,\theta}_{j} = i\} \quad \mbox{for } i \geq \ell,
\end{equation}
be the occupation count of state $i$ for the Markov chain $(\widehat{Q}^{\ell,\theta}_j,~1 \le j \le k)$.
According to Lemma \ref{lemma:PT} $(i)$, for each $k \ge 1$,
\begin{align}
C^{\ell, \theta}_{k}  \ed \widehat{C}^{\ell, \theta}_k &: =  \sum_{j=1}^k 1 \{\widehat{Q}^{\ell,\theta}_j = \widehat{Q}^{\ell,\theta}_{j-1}\}   \label{CQ} \\ 
                              & = \sum_{j=1}^k 1 \{\widehat{Q}^{\ell,\theta}_j = \widehat{Q}^{\ell,\theta}_{j-1} = \ell\} + \sum_{i = \ell+1}^{\infty} \sum_{j=1}^k1 \{\widehat{Q}^{\ell,\theta}_j = \widehat{Q}^{\ell,\theta}_{j-1} = i\} \notag\\
                               &= G^{\ell,\theta}_{\ell}(k) + \sum_{i=\ell+1}^\infty (G^{\ell,\theta}_{i}(k) - 1)^+, \label{C2}
\end{align}
where the last equality follows from the fact that the process $(\widehat{Q}_j^{\ell,\theta}, ~j \ge 0)$ is weakly increasing starting at $\ell$.

\quad Now we establish a connection between the one-parameter chain $\widehat{Q}^{\ell,\theta}$ and a record process. Fix $\ell \in \mathbb{N}_{+}$. For $X_1, X_2, \ldots$ i.i.d.\ with support $\{\ell, \ell+1, \ldots\}$, let $(R_j, ~j \ge 0)$ be the {\em weak ascending record process} of $(X_j, ~j \ge 1)$. That is,
\begin{equation*}
    R_0: = \ell \quad \mbox{and}  \quad R_j : = X_{L_j} \mbox{ for } j \ge 1,
\end{equation*}
where $L_j$ is defined recursively by
\begin{equation*}
    L_1: = 1 \quad \mbox{and} \quad L_{j+1}: = \min\{i> L_j: X_i \ge X_{L_j}\} \mbox{ for } j \ge 1.
\end{equation*}
The sequence $(R_j,~ j \ge 0)$ was first considered by Vervaat \cite{Vervaat}, see also \cite[Section 2.8]{ABN}, and \cite[Lecture 15]{Nevzorov} for further discussion on records of discrete distributions.
It is known that $(R_j,~j \ge 0)$ is a Markov chain with the transition probability function $r(m,n)$ given by
\begin{equation}
\label{rmn}
r(m,n) = \frac{\mathbb{P}(X_1 = n)}{\mathbb{P}(X_1 \ge m)} \quad \mbox{for } m \le n. 
\end{equation}
\begin{proposition}
\label{record}
Let $(R_j,~ j \ge 0)$ be the weak ascending record process of the i.i.d.\ sequence $(X_j,~ j \ge 1)$ with $X_j \stackrel{(d)}{=} \widehat{Q}_1^{\ell,\theta}$; that is,
\begin{equation*}
    \mathbb{P}(X_j = n) = \widehat{q}^{\theta}(\ell, n) \quad \mbox{for } n \ge \ell,
\end{equation*}
where $\widehat{q}^{\theta}$ is defined by \eqref{transitheta}.
Then there is the equality in joint distributions
\begin{equation}
(\widehat{Q}_j^{\ell,\theta}, j \ge 0) \stackrel{(d)}{=}  (R_j,~ j \ge 0).
\end{equation}
\end{proposition}
\begin{proof}
Observe that for $\ell \le m \le n$,
\begin{equation*}
\widehat{q}^{\theta}(\ell,n) = \frac{(\ell)_{n-\ell} \, \theta }{(\theta+\ell)_{n - \ell + 1}}  = \frac{(\ell)_{m-\ell}}{(\theta+\ell)_{m-\ell}} \widehat{q}^{\theta}(m,n).
\end{equation*}
Sum this identity over $n \ge m$ to see that $\mathbb{P}(X_j \ge m) = (\ell)_{m-\ell}/(\theta+\ell)_{m-\ell}$, hence that $\widehat{q}^{\theta}(m, \cdot)$ is the conditional distribution of $X_j$ given $X_j \ge m$, as required.
\end{proof}

\quad It is known \cite[Theorem 1.1]{PY17} that the counts $G^{\ell,\theta}_i(\infty)$ of records at each possible value $i = \ell, \ell+1, \ldots$ are independent and geometrically distributed on $\mathbb{N}_0: = \{0\} \cup \mathbb{N}_{+}$ with parameter $i/(i + \theta)$. 
Combined with Proposition \ref{record}, we get the following result which is a variant of \cite[Proposition 5.1]{GINR}.
\begin{corollary}
\label{gfC}
Let $C^{\ell,\theta}_k$ and $\widehat{C}^{\ell, \theta}_k$ be defined by \eqref{C1} and \eqref{CQ}. Then there is the increasing and almost sure convergence 
\begin{equation*}
    C_k^{\ell, \theta} \ed \widehat{C}^{\ell, \theta}_k  \uparrow C_{\infty}^{\ell,\theta},
\end{equation*}
along with convergence of all positive moments, where the probability generating function of $C^{\ell,\theta}_{\infty}$ is given by
\begin{equation}
\label{pgfC}
F_{\ell,\theta}(z) :=\mathbb{E} z^{C^{\ell,\theta}_{\infty}} = \frac{\Gamma(\ell+1+\theta) \Gamma(\ell+\theta - \theta z)}{\Gamma(\ell) \Gamma(\ell + 1 + 2 \theta - \theta z)}.
\end{equation}
Consequently, the random variable $C^{\ell,\theta}_{\infty}$ has the mixed Poisson distribution with random parameter $-\theta \log H$, where $H$ has the beta$(\ell,\theta+1)$ distribution.
\end{corollary}
This result, combined with Lemma \ref{lemma:PT}$(iii)$, leads to the formula \eqref{PT17formula}:
\[ u_\infty \overset{\theta}{=} \P(C^{1,\theta}_\infty = 0) = F_{1,\theta}(0) = \frac{\Gamma(\theta+2) \Gamma(\theta+1)}{\Gamma(2 \theta+2)}. \]Also note that the random variable $C_{\infty}^{\ell,\theta}$ has a simple representation for $\theta \in \mathbb{N}_{+}$:
\begin{equation}
\label{sumgeo}
C^{\ell,\theta}_{\infty} \stackrel{(d)}{=} \sum_{j = 0}^{\theta} \mathcal{G}_j^{\ell,\theta},
\end{equation}
where $\mathcal{G}_j^{\ell,\theta}$, $0 \le j \le \theta$ are independent and geometrically distributed on $\mathbb{N}_0$ with parameter $(\ell+j)/(\ell+j+\theta)$.

\begin{proof}
The identity \eqref{C2} shows that
\begin{equation*}
     \widehat{C}^{\ell, \theta}_k \uparrow C^{\ell,\theta}_{\infty}: = G^{\ell,\theta}_{\ell}(\infty) + \sum_{i=\ell+1}^\infty (G^{\ell,\theta}_{i}(\infty) - 1)^+ \quad a.s.
\end{equation*}
where $G^{\ell,\theta}_{i}(\infty)$, $i \ge \ell$ are independent and geometrically distributed on $\mathbb{N}_0$ with parameter $p_{i,\theta}: = i/(i + \theta)$. For $G$ geometrically distributed on $\mathbb{N}_0$ with parameter $p$,
\begin{equation*}
    \mathbb{E}z^{G} = \frac{p}{1-(1-p)z} \quad \mbox{and} \quad \mathbb{E}z^{(G-1)^{+}} = p + \frac{(1-p)p}{1-(1-p)z}.
\end{equation*}
As a result,
\begin{align*}
     \mathbb{E} z^{C^{\ell,\theta}_{\infty}} &= \frac{p_{\ell,\theta}}{1-(1-p_{\ell,\theta})z} \prod_{i = \ell + 1}^{\infty} \left(p_{i,\theta} + \frac{(1-p_{i,\theta})p_{i,\theta}}{1-(1-p_{i,\theta})z} \right) \\
    &= \frac{\ell}{\ell+ \theta- \theta z}  \prod_{i = \ell + 1}^{\infty} \frac{i(i+2 \theta - \theta z)}{(i+ \theta)(i+\theta-\theta z)} \\
    & = \frac{\ell}{\ell+ \theta- \theta z} \cdot \frac{\Gamma(\ell+1+\theta) \Gamma(\ell+1+\theta - \theta z)}{\Gamma(\ell+1) \Gamma(\ell + 1 + 2 \theta - \theta z)},
\end{align*}
which leads to the formula \eqref{pgfC}. 
Recall that the generating function of the Poisson$(u)$ distribution is $e^{-u(1-z)}$, and that the Mellin transform of the beta$(p,q)$ variable $H_{p,q}$ is 
\begin{equation*}
\mathbb{E}H_{p,q}^\nu = \frac{\Gamma(\nu+p)\Gamma(p+q)}{\Gamma(p) \Gamma(\nu+p+q)} \quad \mbox{for } \nu > -p.
\end{equation*}
By taking $\nu = \theta(1-z)$, $p = \ell$ and $q = \theta+1$, we identify the distribution of $C^{\ell,\theta}_{\infty}$ with the stated mixed Poisson distribution.
\end{proof}

\quad Let $\psi(x): = \Gamma'(x)/ \Gamma(x)$ be the digamma function, and $\psi^{(k)}(x)$ be the $k^{th}$ derivative of $\psi(x)$. For $k \ge 1$, define
\begin{equation}
\label{Deltakz}
\Delta_{k,\ell,\theta}(z) := \psi^{(k-1)}(\ell+\theta-\theta z) - \psi^{(k-1)}(\ell+1+2\theta-\theta z).
\end{equation}
A simple calculation shows that $F'_{\ell,\theta}(z) = -\theta F_{\ell,\theta}(z) \Delta_{1,\ell,\theta}(z)$ and $\Delta'_{k,\ell,\theta}(z) = -\theta \Delta_{k+1,\ell,\theta}(z)$.
By induction, the derivatives of $F_{\theta}$ can be written as
\begin{equation}
F_{\ell,\theta}^{(k)}(z) = (-\theta)^k F_{\ell,\theta}(z) P_k(\Delta_{1,\ell,\theta}(z), \cdots, \Delta_{k,\ell,\theta}(z)),
\end{equation}
where $P_k(x_1, \cdots, x_k)$ is the {\em $k^{th}$ complete Bell polynomial} \cite[Section 3.3]{Comtet74}.
To illustrate,
\begin{align*}
P_1(x_1) &= x_1,\\
P_2(x_1,x_2) &= x_1^2 + x_2,\\
P_3(x_1,x_2,x_3) &= x_1^3 + 3 x_1 x_2 + x_3,\\
P_4(x_1,x_2,x_3,x_4) &= x_1^4 + 6 x_1^2 x_2 + 4 x_1 x_3 + 3 x_2^2 +x_4,\\
P_5(x_1,x_2,x_3,x_4,x_5) &= x_1^5 + 10 x_1^3 x_2 + 10 x_1^2 x_3 + 15 x_1 x_2^2 + 5 x_1 x_4 + 10 x_2 x_3 + x_5,
\end{align*}
and so on. 
Now by expanding $F_{\ell,\theta}$ into power series at $z=0$ and $z=1$, we get
\begin{equation}
\label{Claw}
\mathbb{P}(C^{\ell,\theta}_\infty = k) = \frac{(-\theta)^k}{k !} \frac{\Gamma(\ell+\theta) \Gamma(\ell+ \theta+1)}{\Gamma(\ell) \Gamma(\ell+ 2 \theta + 1)} P_k(\Delta_{1,\ell,\theta}(0), \cdots, \Delta_{k,\ell,\theta}(0)),
\end{equation}
and
\begin{equation}
\label{binom}
\mathbb{E}\binom{C_\infty^{\ell,\theta}}{k} = \frac{(-\theta)^k}{k !} \, P_k(\Delta_{1,\ell,\theta}(1), \cdots, \Delta_{k,\ell,\theta}(1)),
\end{equation}
where $\Delta_{k,\ell,\theta}(\cdot)$ is defined by \eqref{Deltakz}.
By taking $\theta = 1$ and $k=1$ in \eqref{binom}, we get
\begin{equation}
\mathbb{E}C_{\infty}^{\ell,1} = \psi(\ell+2) - \psi(\ell) = \frac{1+ 2 \ell}{\ell(\ell+1)},
\end{equation}
since $\psi(\ell) = \sum_{j=1}^{\ell-1} 1/j - \gamma$, with $\gamma$ the Euler constant. 

\section{Proof of Proposition \ref{prop:conju}}
\label{sec:prop_proof}

\quad In this section we apply the results of Section \ref{sec:gem_theta_records} to evaluate the renewal sequence $(u_k)$ in the GEM$(1)$ case, and extend to the general GEM$(\theta)$ case. 
The computation boils down to the study of the Markov chain $(\widehat{Q}^{\ell,\theta}_k,~k \ge 0)$ with $\ell = 1$. 
We start by proving Proposition \ref{prop:conju}, corresponding to the case where $\ell = 1$ and $\theta = 1$.
To this end, we need the following duality formula due to Hoffman \cite[Theorem 4.4]{Hoffman} and Zagier \cite[Section 9]{Zagier}.

\begin{lemma}
Let $\zeta(s_1, \cdots, s_k)$ be the multiple zeta value defined by \eqref{multzeta}. Then
\begin{equation}
\zeta(\underbrace{1,\ldots,1}_{k-1}, h+1) = \zeta(\underbrace{1,\ldots,1}_{h-1}, k+1) \quad \mbox{for all } h, k \in \mathbb{N}_{+}.
\end{equation}
In particular, 
\begin{equation}
\label{keyzeta}
\zeta(\underbrace{1,\ldots,1}_{k-2}, 2) = \zeta(k) \quad \mbox{for all } k \geq 2.
\end{equation}
\end{lemma}

\begin{proof}[Proof of Proposition \ref{prop:conju}]
By Lemma \ref{lemma:PT}$(ii)$, for $k \ge 2$,
\begin{align*}
u_k & \stackrel{1}{=} \sum_{1<n_1< \cdots < n_k} \mathbb{P}(\widehat{Q}^{1,1}_1 = n_1, \cdots, \widehat{Q}^{1,1}_k = n_k) \\
& = \sum_{1<n_1< \cdots < n_k} \widehat{q}^{1}(1,n_1) \, \widehat{q}^{1}(n_1,n_2) \cdots \widehat{q}^{1}(n_{k-1},n_k) \\
& = \sum_{1<n_1< \cdots < n_{k-1}} \frac{1}{(n_1+1) \cdots(n_{k-2}+1)(n_{k-1}+1)^2} \\
& = \sum_{0 <n_1< \cdots < n_{k-1}} \frac{1}{(n_1+2) \cdots(n_{k-2}+2)(n_{k-1}+2)^2}.
\end{align*}
For $k \ge 2$ and $x \ge 0$, let
\begin{equation}
\label{Hzeta}
\zeta(\nu_1, \ldots, \nu_{k-1}; x) : =  \sum_{0 <n_1< \cdots < n_{k-1}} \frac{1}{(n_1+x)^{\nu_1} \cdots(n_{k-2}+x)^{\nu_{k-2}}(n_{k-1}+x)^{\nu_{k-1}}},
\end{equation}
be the {\em multiple Hurwitz zeta function} \cite{Olivier}, and $h_k(x): = \zeta(\underbrace{1,\ldots,1}_{k-2},2;x)$. Therefore
\begin{equation}
\label{ukhk}
u_k \stackrel{1}{=} h_k(2) \mbox{ for } k \ge 2.
\end{equation}
We claim that for $k \ge 3$,
\begin{equation}
\label{Hzetarel}
\zeta(\nu_1, \ldots, \nu_{k-1}; x-1) = \zeta(\nu_1, \ldots, \nu_{k-1}; x) + x^{-\nu_{1}} \zeta(\nu_2, \ldots, \nu_{k-1};x)
\end{equation}
In fact,
\begin{align*}
    \zeta(\nu_1, \ldots, \nu_{k-1}; x-1) & = \sum_{0 <n_1< \cdots < n_{k-1}} \frac{1}{(n_1+x-1)^{\nu_1} \cdots(n_{k-2}+x-1)^{\nu_{k-2}}(n_{k-1}+x-1)^{\nu_{k-1}}} \\
     & = \sum_{0  \leq n_1< \cdots < n_{k-1}} \frac{1}{(n_1+x)^{\nu_1} \cdots(n_{k-2}+x)^{\nu_{k-2}}(n_{k-1}+x)^{\nu_{k-1}}}, 
\end{align*}
and writing this expression as two sums over the distinct sets $\{0 = n_1 < n_2 < \cdots < n_{k-1}\}$ and $\{0 < n_1 < n_2 < \cdots < n_{k-1}\}$ yields the formula \eqref{Hzetarel}. Consequently,
\begin{equation}
\label{recur}
h_k(x-1) = h_k(x) +  x^{-1} h_{k-1}(x) \quad  \mbox{for } k \ge 3 .
\end{equation}
By taking $x = 2$ and $x = 1$ in \eqref{recur}, we get for $k \ge 3$,
\begin{equation*}
h_k(1) = h_k(2) +  \frac{1}{2} h_{k-1}(2) \quad \mbox{and} \quad h_k(0) = h_k(1) +  h_{k-1}(1),
\end{equation*}
which implies that for $k \ge 4$,
\begin{equation}
\label{recur2}
h_k(0) = h_k(2) + \frac{3}{2} h_{k-1}(2) + \frac{1}{2} h_{k-2}(2).
\end{equation}
According to the formula \eqref{keyzeta},
\begin{equation}
\label{multzetafor} 
h_k(0) = \zeta(\underbrace{1,\ldots,1}_{k-2}, 2) = \zeta(k).
\end{equation}
By \eqref{ukhk}, \eqref{recur2} and \eqref{multzetafor}, we derive the recursion \eqref{rec} for $k \ge 4$.
Recall that by definition, we have $u_0=1$ and it is easy to check that $u_1\overset{1}{=}1/2$.
By symbolic integration, we get:
\begin{equation*}
u_2 \overset{1}{=} -\frac{5}{4} + \zeta(2) \quad \mbox{and} \quad u_3 \overset{1}{=} \frac{13}{8} - \frac{3}{2} \zeta(2) + \zeta(3),
\end{equation*}
which satisfies the recursion for $k = 2,3$.
So the part $(i)$ of the proposition is proved.
The equivalences $(i) \Leftrightarrow (ii) \Leftrightarrow (iv)$ are straightforward, and $(ii) \Leftrightarrow (iii)$ follow by partial fraction decomposition.
We will see in Section \ref{sec:renewal_seq} that the parts $(i)$, $(iv)$ in Proposition \ref{prop:conju} are valid for general recursions of the form $au_{k-2}+bu_{k-1}+cu_k = \zeta(k)$.
\end{proof}

\quad In the sequel, we aim to extend the above calculation to general $\theta>0$. It is easily seen that
\begin{align*}
    u_k & \stackrel{\theta}{=} \sum_{1<n_1< \cdots < n_k} \mathbb{P}(\widehat{Q}^{1,\theta}_1 = n_1, \cdots, \widehat{Q}^{1,\theta}_k = n_k) \\
        & = \sum_{1<n_1< \cdots < n_k} \frac{\theta^k \, (n_k-1)!}{(\theta+ n_1) \cdots (\theta+n_{k-1}) (\theta + 1)_{n_k}} \\ 
        & = \theta^k \, \sum_{0 <n_1< \cdots < n_{k-1}} \frac{1}{(\theta+ n_1+1) \cdots (\theta+n_{k-1}+1)} \sum_{n_k > n_{k-1}} \frac{n_k!}{(\theta+1)_{n_k+1}}.
\end{align*}
Note that for all $k \ge 0$,
\begin{align*}
    \sum_{n \ge k} \frac{n!}{(\theta+1)_{n+1}} &= \sum_{n \ge k} \frac{1}{\theta} \frac{(\theta + n+1)n!-(n+1)!}{(\theta+1)_{n+1}}  \\
    &= \sum_{n \ge k} \frac{1}{\theta}\left (\frac{n!}{(\theta+1)_{n}} - \frac{(n+1)!}{(\theta+1)_{n+1}} \right ) \\
    &= \frac{k!}{\theta\,(\theta+1)_{k}} = \frac{\Gamma(\theta) \, \Gamma(k+1)}{\Gamma(\theta+k+1)}.
\end{align*}
Therefore,
\begin{equation}
\label{eq:uktheta}
u_k \stackrel{\theta}{=} \theta^k \Gamma(\theta) \sum_{0 <n_1< \cdots < n_{k-1}} \frac{1}{(\theta+n_1+1) \cdots (\theta+n_{k-1}+1)}\, \frac{\Gamma(n_{k-1}+2)}{\Gamma(n_{k-1}+\theta+2)}.
\end{equation}
It seems to be difficult to simplify the expression \eqref{eq:uktheta} for general $\theta$. We focus on the case where $\theta \in \mathbb{N}_{+}$. Let
\begin{equation*}
    h_{k,\theta}(x): = \theta^k \Gamma(\theta)\sum_{0 <n_1< \cdots < n_{k-1}} \frac{\Gamma(n_{k-1}+1-\theta + x)}{(n_1+x) \cdots(n_{k-1}+x) \Gamma(n_{k-1}+1+x)},
\end{equation*}
so $u_k \stackrel{\theta}{=} h_{k,\theta}(\theta+1)$.
Again it is elementary to show that
\begin{equation*}
    h_{k,\theta}(x-1) = h_{k,\theta}(x) + \frac{\theta}{x} h_{k-1,\theta}(x).
\end{equation*}
Consequently, the sequence $(u_k,~ k \ge 0)$ satisfies a $(\theta+1)$-order recursion:

 \begin{equation}
 u_k + a_{1,\theta} u_{k-1} + \cdots + a_{\theta+1,\theta} u_{k-\theta-1} = h_{k,\theta}(0),
 \end{equation}
where
\begin{equation*}
a_{i,\theta}: = \sum_{0<n_1 < \cdots < n_i \le \theta+1} \frac{\theta^i}{n_1 \cdots n_i} \quad \mbox{for } 1 \le i
\le \theta+1,
\end{equation*}
and $h_{k,\theta}(0)$ is a variant of the multiple Hurwitz zeta function.
\section{Renewal sequences derived from the zeta function}
\label{sec:renewal_seq}

\quad Look at the sequence 
\begin{equation}
\label{ukdef}
u_k := \sum_{n = 1}^\infty \frac{n^{-k}}{q(n)}   \quad ( k = 0,1, 2, \ldots)
\end{equation}
where 
\begin{equation}
\label{qdef}
q(n) := a n^2 + b n  + c
\end{equation}
is a generic quadratic function of $n$. We are interested in conditions on $q$ which allow
the sequence $(u_k, k = 0,1, \ldots)$ to be interpreted as a renewal sequence \cite{Feller}.
Basic requirements are that $q(n) > 0 $ for all $n = 1,2, \ldots$, so at least $a > 0$, and that $u_0 = 1$, which is
a matter of normalization of coefficients of $q$. The sequence $1/q(n)$, $n = 1,2, \ldots$ then defines a probability
distribution on the positive integers. If $X$ denotes a random variable with this distribution, so $\P(X= n) = 1/q(n)$,
$n = 1,2, \ldots$,
then  \eqref{ukdef} becomes
\begin{equation}
\label{ukmom}
u_k = \E (1/X)^k  \quad ( k = 0,1, 2, \ldots).
\end{equation}
That is to say, $u_k$ is the $k^{th}$ moment of the probability distribution of $1/X$ on $[0,1]$.
Obviously, $0 \le u_k \le 1$, and by the Cauchy-Schwartz inequality applied to 
$$(1/X)^k = (1/X)^{(k-1)/2} (1/X)^{(k+ 1)/2},$$
\begin{equation}
\label{kaluza}
u_k ^2 \le u_{k-1} u_{k+1}  \qquad ( k = 1,2, \ldots).
\end{equation}
A  sequence $(u_k)$ bounded between $0$ and $1$ with $u_0 = 1$ and subject to \eqref{kaluza} is called a {\em Kaluza sequence} \cite{Kaluza}.
By a classical theorem of Kaluza, every such sequence is a {\em renewal sequence} \cite{Kaluza}.
See \cite{PT17} for an elementary proof and further references.
In view of Proposition $1.2$, we are motivated to study such renewal sequences $(u_k)$ and
the associated distribution $(f_k)$ of the time until first renewal, whose generating functions
\begin{equation}
\label{gfs}
U(z):= \sum_{k=0}^\infty u_k z^k \mbox{ and } F(z):= \sum_{k=1}^\infty f_k z^k  \qquad ( | z | < 1 )
\end{equation}
are known \cite{Feller} to be related by
\begin{equation}
\label{gfrel}
U(z)= (1- F(z))^{-1} \mbox{   and } F(z)  = 1 - U(z)^{-1} .
\end{equation}
This identity of generating functions corresponds to the basic relation 
\begin{equation}
\label{ufrec}
u_k = f_k + f_{k-1} u_1 + \cdots + f_1 u_{k-1}  \qquad ( k = 1, 2, \ldots)
\end{equation}
which allows either of the sequences $(u_k)$ and $(f_k)$ to be derived from the other.
Observe that the definition $q(n) = a n^2 + b n + c$ gives
\begin{equation}
\label{basicid}
\frac{c n^{-k}}{q(n)} + \frac{b n^{-(k-1)}}{q(n)} + \frac{a n^{-(k-2)}}{q(n)} = n^{-k}
\end{equation}
and hence, for $k \ge 2$,
\begin{equation}
\label{urec}
c u_k + b u_{k-1} + a u_{k-2} = \sum_{n=1}^\infty n^{-k} = \zeta(k).
\end{equation}
It follows that $U(z)$ and hence $F(z)$ can always be expressed in terms of the well known (see \cite[formula 6.3.14]{abramowitz1964handbook}) generating
function of $\zeta$ values
\begin{equation}
G(z):= \sum_{n=2}^\infty \zeta(n) z^n = - z ( \gamma + \psi(1-z) ) , \qquad ( | z | < 1 )
\end{equation}
where $\gamma$ is Euler's constant and $\psi(x):= \Gamma'(x)/\Gamma(x)$ is the digamma function,
as
$$
c( U(z) - u_0 - u_1 z ) + b z ( U(z) - u_0 ) + a z^2 U(z) = G(z)
$$
or 
$$
q(z) U(z) - c ( u_0 + u_1 z ) - u_0 b z  = G(z)
$$
which rearranges as
\begin{equation}
\label{ufromG}
U(z) = \frac{ c u_0 + ( b u_0 + c u_1 ) z + G(z) }{ q(z) }.
\end{equation}

Defining $r_1,r_2\in \mathbb{C}$ as the two roots of $q$, we have
\begin{gather*}
q(z) = a(z-r_1)(z-r_2)\\
b = -a(r_1+r_2) \text{ and } c = a r_1 r_2.
\end{gather*}
Note that our assumption that $q(n)>0$ for all $n=1,2,\ldots$ implies that the roots $r_1$ and $r_2$ are not positive integers.
A straight-forward computation shows that the condition $u_0=1$ implies
\[ a = \begin{cases}
\dfrac{\psi(1-r_2)-\psi(1-r_1)}{r_1 - r_2} &\quad \text{if } r_1 \neq r_2\\
\psi'(1-r_1) &\quad \text{if }r_1 = r_2,
\end{cases} \]
and that we have
\[ u_1= \frac{1}{2c}\left (-b+2\gamma+\psi(1-r_1)+\psi(1-r_2)\right ). \]
Finally, obtaining $F(z)$ from \eqref{gfrel} and \eqref{ufromG} and taking derivatives gives us
\begin{align}
&F'(1)=q(1),\\
&\begin{aligned}
F''(1)&+q(1)(1-q(1)) \\
&= (4-q(1))a-q(-1)+q(1)(c(1+2u_1)+b)\\
&= (4-q(1))a-q(-1)+q(1)(c+2\gamma+\psi(1-r_1)+\psi(1-r_2)).
\end{aligned}
\end{align}

To summarize, and combine with some standard renewal theory:
\begin{proposition} \label{prop:ukfromq}
	Let $q(n)$ be any quadratic function of $n = 1,2, \ldots$ with $q(n) > 0$ for all $n$, normalized so that
	$
	u_0:= \sum_{n = 1}^\infty 1/q(n) = 1,
	$
	and let $u_k:= \sum_{n = 1}^\infty n^{-k}/ q(n)$ for $k \ge 1$. Then $(u_k)$ is a decreasing, positive recurrent renewal sequence, with
	$$
	\lim_{k \to \infty} u_k = 1/q(1) .
	$$
	The corresponding distribution of an i.i.d.\ sequence of positive integer valued random variables $Y_1, Y_2, \ldots$ with
	$\P(Y_1 + \ldots + Y_m = k \mbox{ for some m } ) = u_k$ has distribution with mean and variance
	\begin{gather*}
    \E (Y_1) = q(1)\\
    \Var(Y_1) = (4-q(1))a-q(-1)+q(1)(c+2\gamma+\psi(1-r_1)+\psi(1-r_2))
    \end{gather*}
	and probability generating function $F(z):= \E (z^{Y_1} )$ given by \eqref{gfrel} for $U(z)$ as in \eqref{ufromG}.
\end{proposition}

{\bf Example.}
Take $a = 1/2, b = 3/2, c = 1$. Then $q(n) = (n+1)(n+2)/2$ makes $u_0 = 1$, $1/u_\infty = \E(Y_1) = q(1) = 3$, and $\Var(Y_1) = 11$.
From Proposition \ref{prop:conju}, we know that this sequence $(u_k)$ is the renewal sequence associated with a GEM$(1)$ random partition of $[0,1]$.
Equivalently in the terminology of random permutations \cite{PT17}, $(u_k)$ is the renewal sequence of the \emph{splitting times} of a GEM$(1)$-biased permutation.
Therefore $Y_1$ is distributed as $T_1$, the first splitting time of $\Pi$ a GEM$(1)$-biased permutation.
In particular, we have $\E(T_1) = 3$, $\Var(T_1) = 11$.
\section{Development of the GEM\texorpdfstring{$(1)$}{(1)} case}
\label{sec:gem_one}

\quad The Markov chain $(\widehat{Q}_k)$ described in Lemma \ref{lemma:PT}, with uniform stick-breaking factors -- i.e.\ in the GEM$(1)$ case -- was first studied by Erdős, Rényi and Szüsz \cite{ERS}, where it appears as Engel's series derived from $U$ a uniform random variable on $(0,1)$.
More precisely, if $2 \leq q_1 \leq q_2 \leq \ldots$ is the unique random sequence of integers such that 
\[ U = \frac{1}{q_1}+\frac{1}{q_1 q_2} +\cdots+\frac{1}{q_1 q_2 \cdots q_n} +\cdots, \]
then we have 
\[ (q_i - 1,\, i\geq 1) \ed (\widehat{Q}_i,\, i\geq 1), \]
where we condition on $\widehat{Q}_0 = 1$.

\quad Here we give explicit formulas for the distribution of $\widehat{Q}_k$ in terms of iterated harmonic sums and the Riemann zeta function.
The transition probabilities $\widehat{q}(m,n) := \P(\widehat{Q}_{k+1} = n \mid \widehat{Q}_k = m)$ are given by
\[ \widehat{q}(m,n) = \frac{m}{n(n+1)}. \]
Then the joint probability distribution of $\widehat{Q}_1, \ldots , \widehat{Q}_k$ is given by the formula
\begin{align}
\P( \widehat{Q}_{1} = n_1, \ldots, \widehat{Q}_{k-1} = n_{k-1}, \widehat{Q}_k = n_k) = \frac {1(1\le n_1 \le \ldots \le n_{k-1} \le n_k)}{(n_1 +1)\cdots (n_{k-1}+1) (n_k+1) n_k }  .
\label{vist}
\end{align}
It follows that for $k = 1,2, \ldots$
\begin{align}
\P(\widehat{Q}_k = n) &= \frac{1}{n(n+1)} \sum_{1\leq n_1\leq \ldots \leq n_{k-1}\leq n} \frac{1}{(n_1+1)\cdots (n_{k-1}+1)}\notag \\
& =\frac{\SH_{k-1}(n+1) - \SH_{k-2}(n+1)}{n(n+1)}. \label{vkdist}
\end{align}
%

where $\SH_{-1}(n) = 0, \SH_{0}(n) = 1$, and $\SH_k(n)$ for $k = 1, 2, \ldots$ is the $k$th {\em iterated harmonic sum} defined by 
\begin{align}
\label{SH}
\SH_k(n):= \sum_{m=1}^n \frac{ \SH_{k-1}(m)  }{m} 
= \sum_{1 \le n_1 \le n_2 \le \cdots \le n_k \le n} \frac{ 1 }{ n_1 n_2 \cdots n_k } .
\end{align}
In particular, $\SH_1(n) = \Hn$ is the $n$-th harmonic number, and $\SH_2(n) = \sum_{m=1}^n H(m)/m$. Such {\em iterated} or {\em multiple} harmonic
sums have attracted the attention of a number of authors \cite{BBG,AK,AKO}.  Since \eqref{vkdist} describes a
probability distribution over $n \in \{1,2,3,\ldots\}$, we deduce by induction that
\begin{align}
\sum_{n=1}^ \infty \frac{ \SH_k(n+1)}{n(n+1)}  =  k + 1      \quad (k  \ge 0 ).
\label{hsumk}
\end{align}
This identity has the probabilistic interpretation that for each $k = 1,2, \ldots$, $n = 1,2, \ldots$
\begin{align}
\label{SHinterp}
\E \left[ \sum_{j=1}^k 1( \widehat{Q}_j = n) \right] = \sum_{j=1}^k \P( \widehat{Q}_j = n)  = \frac{ \SH_{k-1} (n+1) } { n (n+1) },
\end{align}
where $\sum_{j=1}^k 1( \widehat{Q}_j = n)$ is the number of times $j$ with $1 \le j \le k$ that $\widehat{Q}_j$ has value $n$.
It is easily seen that for each fixed $n$ the sequence $\SH_{k-1} (n)$ is increasing with limit $n$ as $k \te \infty$. So the limit version of
\eqref{SHinterp} is for $n = 1,2, \ldots$
\begin{align}
\label{liminterp}
\E \left[ \sum_{j=1}^\infty 1( \widehat{Q}_j = n) \right] = \sum_{j=1}^\infty \P( \widehat{Q}_j = n)  = \frac{1}{n}.
\end{align}
As observed by R\'enyi \cite{renyi}, the random variables $S_1, S_2, \ldots$, with
$S_n:= \sum_{j=1}^\infty 1( \widehat{Q}_j = n) $, are independent with the geometric distribution on $\{0,1,2, \ldots\}$ with
parameter $1 - 1/(n+1)$:
\begin{align}
\label{geom}
\P(S_n = s) = \frac{1}{(n+1)^s} \left( 1 -  \frac{1}{n+1} \right) \quad (n = 1,2, \ldots,\, s = 0,1, \ldots)
\end{align}
which also implies \eqref{liminterp}.

\quad Some known results about the iterated harmonic sums $\SH_k(n)$ can now be interpreted as features of the distributions of $\widehat{Q}_k$.
Arakawa and Kaneko \cite{AK}
defined the function
\begin{align}
\xi_m(s) = \frac{1}{\Gamma(s)} \int_0^\infty t^{s-1} e^{-t} \frac{ \Li_m(1-e^{-t})}{ (1 - e^{-t} ) } dt ,
\label{arakan}
\end{align}
where $\Li_m$ is the polylogarithm function 
\[ \Li_m(z):= \sum_{n=1}^\infty z^n/n^m. \]
The integral converges for $\Re(s) >0 $
and the function $\xi_m(s)$ continues to an entire function of $s$. They showed that values of $\xi_m(s)$ for
$s$ can be expressed in terms of multiple zeta values, and observed in particular that 
$$
\xi_1(s) = s \zeta(s+1),
$$
which can readily be derived from \eqref{arakan} using the identity $ \Li_1(1-e^{-t}) =  t$.
Ohno \cite{Ohno} 
then showed that for positive integer $m$ and $k$:
\begin{align}
\xi_m(k) = \sum_{1 \le n_1 \le \cdots \le n_k} \frac{ 1 }{ n_1 n_2 \cdots n_{k-1} n_k^{m+1} } = \sum_{n=1}^\infty \frac{ \SH_{k-1}(n) }{n^{m+1}}.
\label{ohno1}
\end{align}
That is, with the replacement $k \rightarrow k+1$ and taking $ m=1 $,
\begin{align}
\sum_{n=1}^\infty \frac{\SH_k(n)}{n^2 } = (k+1) \zeta(k+2)    \quad (k= 0,1,2, \ldots)
\label{ohno3}
\end{align}
Subtracting $1$ from both sides of this identity  gives a corresponding formula with summation from
$n=2$ to $\infty$ on the left. Comparing with the more elementary formula \eqref{hsumk}, and using
$$
\frac{1}{n} - \frac{1}{n+1} = \frac{1}{n(n+1)},
$$
it follows that
\begin{align}
\sum_{n=1}^\infty \frac{\SH_k(n+1)}{n (n+1)^2} = k + 2 - (k+1) \zeta(k+2)    \quad (k = 0,1,2, \ldots)
\label{ohno4}
\end{align}
Plugged into formula \eqref{vkdist} for the distribution of $\widehat{Q}_k$, this gives a formula for the first inverse
moment of $\widehat{Q}_k+1$:
\begin{align}
\E\left(\frac{1}{\widehat{Q}_k+1}\right) = 1 - k \zeta(k+1) + (k-1) \zeta(k).
\label{ohno5}
\end{align}

\quad By this stage, we have reached some identities for multiple zeta values which cannot be easily verified symbolically using Mathematica, though
they are readily checked for modest values of $k$ to limits of numerical precision.
The case $k = 1$ of \eqref{ohno3} reduces to
$$
\sum_{n=1}^\infty \frac{H(n)}{n^2} = 2 \zeta(3),
$$
which Borwein et al.\ \cite{BBG} attribute to Euler, and which can be confirmed symbolically on Mathematica.
The case $k = 2$ of \eqref{ohno3} expands to
\begin{equation} \label{2sum}
\sum_{n=1}^\infty \frac{\SH_2(n)}{n^2} = 
\frac{1}{2} \sum_{n=1}^\infty \frac{H^2(n)}{n^2} 
+ \frac{1}{2} \sum_{n=1}^\infty \frac{H_2(n)}{n^2}  = 3 \zeta(4),
\end{equation}
with $H_k(n) = \sum_{m=1}^{n}1/m^k$, which can also be confirmed symbolically on Mathematica. The decomposition into power sums is found to be
\begin{align}
\frac{1}{2} \sum_{n=1}^\infty \frac{H^2(n)}{n^2}  &= \frac{17}{8} \zeta(4)   \label{11}, \\
\frac{1}{2} \sum_{n=1}^\infty \frac{H_2(n)}{n^2}  &= \frac{7}{8}  \zeta(4)  \label{2},
\end{align}
and summing these two identities yields \eqref{2sum} as required.
According to Borwein and Borwein \cite{borwein_intriguing_1995}, formula \eqref{11} was first discovered numerically by Enrico Au-Yeun, and provided impetus
to the surge of effort at simplification of multiple Euler sums by Borwein and coauthors \cite{bailey_experimental_1994}. Borwein and Borwein's first
rigorous proof of \eqref{11} was based on the integral identity
$$
\frac{1} {\pi} \int_0^\pi \theta^2 \log^2 ( 2 \cos \theta/2) d \theta = \frac{11}{2} \zeta(4),
$$
which they derived with Fourier analysis using Parseval's formula. Later \cite{BBG}, they gave a systematic account of
evaluations of multiple harmonic sums, including \eqref{11} as an exemplar case. In particular, they made a systematic study of the Euler sums
\begin{align}
s_h(k,s):= \sum_{n=1}^\infty \frac{ H^k(n) }{ (n+1)^s },
\end{align}
and 
\begin{align}
\sigma_h(k,s):= \sum_{n=1}^\infty \frac{ H_k(n) }{ (n+1)^s },
\end{align}
and a number of other similar sums. They proved a number of exact reductions of such sums to evaluations of the zeta function at integer arguments and
established other such reductions beyond a reasonable doubt by numerical computation.

\quad With this development, we can prove the following result for $C_k^{1,1}$ the number of empty intervals among the first $k$ intervals at the stopping time $n(k,1)$
when the first sample point falls outside the union $[0,R_k)$ of these intervals.
\begin{proposition}
Let $C_k^{1,1}$ be defined by \eqref{C1} for $\ell = 1$ and $\theta = 1$. Then
\begin{equation}
\label{meanempty}
\mathbb{E}C_{k}^{1,1} =\left\{ \begin{array}{ccl} 
    1/2 + k - (k-1) \zeta(k) & \mbox{for } k \ge 2, \\
    1/2 & \mbox{for }k = 1,
    \end{array}\right.
\end{equation}
\end{proposition}
\begin{proof}
According to \eqref{CQ},
\begin{equation*}
    \mathbb{E}C_k^{1,1} = \sum_{j=1}^k \mathbb{P}(\widehat{Q}_j^{1,1} = \widehat{Q}_{j-1}^{1,1}),
\end{equation*}
where for $j \ge 2$ we use \eqref{vkdist} to evaluate $\mathbb{P}(\widehat{Q}_j^{1,1} = n)$. 
So for $j \ge 2$,
\begin{align*}
    \mathbb{P}(\widehat{Q}_j^{1,1} = \widehat{Q}_{j-1}^{1,1}) & = \sum_{n \ge 1} \mathbb{P}(\widehat{Q}_j^{1,1} = \widehat{Q}_{j-1}^{1,1} = n) \\
    & = \sum_{n \ge 1} \mathbb{P}(\widehat{Q}^{1,1}_{j-1} = n ) \widehat{q}^{1}(n,n) \\
    & = \sum_{n \ge 1}\frac{H_{j-2}^{*}(n+1) - H_{j-3}^{*}(n+1)}{n(n+1)^2} \\
    & = \left\{ \begin{array}{ccl} 
    1 - (j-1) \zeta(j) + (j-2) \zeta(j-1) & \mbox{for } j \ge 3, \\
    2 - \zeta(2) & \mbox{for }j = 2,
    \end{array}\right.
\end{align*}
where the last equality follows from \eqref{ohno4}. Also note that
\begin{equation*}
    \mathbb{P}(\widehat{Q}_1^{1,1} = \widehat{Q}_{0}^{1,1}) = \widehat{q}^1(1,1) = 1/2.
\end{equation*}
The formulae \eqref{meanempty} follow from the above computations.
\end{proof}

\quad Recall the interpretation of the binomial moments $\mathbb{E}\binom{C^{1,1}_k}{j}$ from \cite[(6.2)]{PT17}. The case $j = 1$ has been evaluated in \eqref{meanempty}. 
For $j = k$,
\begin{equation}
\mathbb{E} \binom{C_k^{1,1}}{k} = \mathbb{P}(C_k^{1,1} = k) = \frac{1}{2^k},
\end{equation}
and for $j = k-1$,
$$\mathbb{E} \binom{C_k^{1,1}}{k-1} = \frac{k}{2^k} + \mathbb{P}(C_k^{1,1} = k-1).$$
Note that
\begin{align*}
\mathbb{P}(C_k^{1,1} = k-1) &= \sum_{m=0}^{k-1} \sum_{n > 1} \widehat{q}^{1}(1,1)^m \widehat{q}^{1}(1,n) \widehat{q}^{1}(n,n)^{k-1-m} \\
                                              & = \frac{1}{2^{k-1}} - 2 \sum_{n \ge 1} \frac{1}{n(n+1)(n+2)^k}
\end{align*}
By partial fraction decomposition,  
$$ \frac{1}{n(n+1)(n+2)^k} = \frac{1}{2^k n} - \frac{1}{n+1} + \frac{2^k-1}{2^k(n+2)} + \sum_{j = 2}^k \frac{2^{k+1-j} - 1}{2^{k+1-j}(n+2)^j},$$
which leads to
$$
\sum_{n \ge 1}  \frac{1}{n(n+1)(n+2)^k}   = \frac{k}{2^{k+1}} - k + 1 + \sum_{j=2}^k \frac{2^{k+1-j} - 1}{2^{k+1-j}} \zeta(j).
$$
Therefore,
\begin{equation}
\label{k1mom}
\mathbb{E} \binom{C_k^{1,1}}{k-1}
= 2k - 2 + \frac{1}{2^{k-1}} - \sum_{j=2}^k \frac{2^{k+1-j} - 1}{2^{k-j}} \zeta(j)
\end{equation}
We have proved in Corollary \ref{gfC} that the random variables $C^{1,1}_k$ converges in distribution as $k \rightarrow \infty$. 
Consequently,
$$\mathbb{E} \binom{C_k^{1,1}}{k-1} \longrightarrow 0 \quad \mbox{as } k \rightarrow \infty.$$
It is easily seen from the expression  \eqref{k1mom} that this is equivalent to the well known formula
$$\sum_{n = 2}^\infty ( \zeta(n) - 1 ) = 1.$$
But the formulas for other binomial moments seem to be difficult, even for $j = 2$.
Generally, we are interested in the exact distribution of $C^{1,1}_k$ on $\{0,1,...k\}$ for $k = 1,2, \ldots$ Simple formulas are found for $k =1, 2,3, 4$ as displayed in the following table. \\

Table of $\mathbb{P}(C^{1,1}_k = j)$ with $0 \le j \le k$ for $k = 1,2,3,4$.
\begin{center}
\small
\begin{tabular}{ c | c  c  c  c  c  c c}
\multicolumn{1}{l}{$k$} &&&&&&&\\\cline{1-1} 
&&&&& \\
1 &$\frac{1}{2}$& $\frac{1}{2}$ & & & & \\
&&&&& \\
2 &$ - \frac{5}{4} + \zeta(2)$&$2 -\zeta(2)$& $\frac{1}{4}$& && \\
&&&&& \\
3 &$\frac{13}{8} - \frac{3}{2} \zeta(2)+\zeta(3) $&$-\frac{37}{8} + 3 \zeta(2)$&$\frac{31}{8} - \frac{3}{2}\zeta(2) - \zeta(3)$&$\frac{1}{8}$&&\\
&&&&& \\
4 &$-\frac{29}{16} + \frac{7}{4} \zeta(2) - \frac{3}{2} \zeta(3) + \zeta(4)$ &$\frac{57}{8} -\frac{21}{4} \zeta(2) + \frac{3}{2} \zeta(3)$&$-\frac{41}{4} + \frac{21}{4} \zeta(2) + \frac{3}{2} \zeta(3)$& $\frac{47}{8} - \frac{7}{4} \zeta(2) - \frac{3}{2}\zeta(3) -\zeta(4) $&$\frac{1}{16}$& \\
&&&&& \\ \hline
\multicolumn{1}{l}{} &0&1&2&3&4& ~$j$\\
\end{tabular}
\end{center}

The details are left to the reader. Observe that up to $k = 4$,  all of the point probabilities in the distribution of  $C_k^{1,1}$ are rational linear combinations of zeta values. This is also true with $j = 0, k-1, k$ for all $k$.
We leave open the problem of finding an explicit formula for $\P( C_k^{1,1} = j)$ for general $j$ and $k$, but make the following conjecture:
\begin{conj}
For each $k \ge 1$ and $0 \le j \le k$, 
$$\mathbb{P}(C^{1,1}_k = j) = q_{k,1} + \sum_{j = 2}^k q_{k,j} \zeta(j), $$
with $q_{k,j}$ rational numbers.
\end{conj}
\section{Evaluation of $u_{2:n}$ and its limit}
\label{sec:u2}

\quad In this section we derive an explicit formula for $u_2$ in the general GEM$(\theta)$ case, which is based on evaluation of the combinatorial expressions of $u_{2:n}$ given later in \eqref{5kinds}.
In principle, the analysis of $u_{2:n}$ can be extended to $u_{k:n}$ for $k \ge 3$, but there will be an annoying proliferation of cases. Already for $k = 2$, it requires considerable care not to overcount or undercount the cases. 

\quad Let $\Pi_n$ be the partition of $[n]$ generated by a random permutation of $[n]$, or, more generally, by any consistent sequence of exchangeable random partitions of $[n]$, with {\em exchangeable partition probability function (EPPF)} $p$. 
See \cite{pitmanbook} for background. 
The function $p$ is a function of compositions $(n_1, \ldots, n_k)$
of $n$, 
which gives for every $m \ge n$ the
probability $p(n_1, \ldots, n_k)$
that for each particular listing of elements of $[n]$ by a permutation, say $(x_1, \ldots, x_n)$, that the first $n_1$ elements
$\{x_1, \ldots, x_{n_1} \}$ fall in one block of $\Pi_m$, and 
if $n_1 < n$ the next $n_2$ elements $\{x_{n_1 + 1}, \ldots, x_{n_1 + n_2} \}$ fall in another block of $\Pi_m$, 
and if $n_1 + n_2 < n$ the next $n_3$ elements $\{x_{n_1 + n_2 + 1}, \ldots, x_{n_1 + n_2 + n_3} \}$ fall in a third block of $\Pi_m$, 
and so on.  In other words, $p(n_1, \ldots, n_k)$ is the common probability, for every $m \ge n$, that the restriction of $\Pi_m$ to $[n]$ equals any particular partition of $[n]$ whose blocks are of sizes $n_1, \ldots , n_k$.
For the Ewens $(\theta)$ model, there is the well known formula
$$
p_\theta( n_1, \ldots, n_k ) = \frac{ \theta ^{k-1} } { ( 1 + \theta ) _{n-1} } \prod_{i=1}^k  (1)_{n_i - 1}, \mbox{ where } n = n_1 + \cdots + n_k.
$$
This formula for $\theta = 1$ 
$$
p_1( n_1, \ldots, n_k ) = \frac{ \prod_{i=1}^k  (n_i - 1)!   } {n! }, \mbox{ where } n = n_1 + \cdots + n_k,
$$
corresponds to the case when $\Pi_n$ is the partition of $[n]$ generated by the cycles of a uniformly distributed random permutation of $[n]$. 
Then the denominator $n!$ is the number of permutations of $[n]$, while the product in the numerator is the obvious enumeration of the number of permutations of $n$
in which $[n_1]$ forms one cycle, and  $[n_1 + n_2 ] \setminus [n_1]$ forms a second cycle, and so on. Essential for following arguments is the less obvious
{\em consistency property} of uniform random permutations, that $p_1(n_1, \ldots, n_k)$ is also, for every $m \ge n$, the probability that 
$[n_1]$ is the restriction to $[n]$ of one cycle of $\Pi_m$, and $[n_1 + n_2 ] \setminus [n_1]$ the restriction to $[n]$ of a second cycle of $\Pi_m$, and so on. 
This basic consistency property of random permutations allows the sequence of random partitions $\Pi_n$ of $[n]$ to be constructed according to the {\em Chinese Restaurant Process}, so the restriction of $\Pi_m$ to $[n]$ is $\Pi_n$ for every $n < m$.

\quad Consider first for $n \ge 2$ the probability that the same block of $\Pi_n$ is discovered first in examining elements of $[n]$ from left to right as in examining elements of $[n]$ from right to left.
This is the probability that $1$ and $n$ fall in the same block of $\Pi_n$. By exchangeability, this is the same as the probability that $1$ and $2$ fall in the same block, that is
\begin{equation}
u_{1:n}:= \P ( \mbox{$1$ and $n$ in the same block} ) =  p(2) \eqth  p_\theta(2) = \frac{ 1 } { 1 + \theta }.  
\end{equation}
Next, consider for $n \ge 3$ the probability $u_{2:n}$ that the union of the first two blocks found in sampling left to right equals 
the union of the first two blocks found in sampling right to left.  
\begin{proposition}
For each $n \ge 3$, and each exchangeable random partition $\Pi_n$ of $[n]$ with EPPF $p$,
\begin{equation}
\label{5kinds}
\begin{aligned}
u_{2:n} &= p(n) + (n-2) p(n-1,1) + \sum_{1 < j < k < n } p( j-1 + n - k , 2 ) \\
&\quad + \sum_{j=1}^{n-1} p(j,n-j) + \sum_{1 < j < k < n } p(j, n-k+1) .
\end{aligned}\end{equation}
\end{proposition}
\begin{proof}
The terms in \eqref{5kinds} are accounted for as follows:
\begin{itemize}
\item $p(n)$ is the probability that there is only one block, with $1$ and $n$ in this block.
\item $(n-2) p(n-1,1)$ is the probability that there are only two blocks, with both $1$ and $n$ in this block, while the second block a singleton, which may be
$\{j\}$ for any one of the $n-2$ elements $j \in \{ 2, \ldots, n-1 \}$.
\item $\sum_{1 < j < k < n } p( j-1 + n - k , 2 ) $ is the sum of the probabilities that there are two or more blocks, with both $1$ and $n$ in this block, with $j$ the
first element and $k$ the last element of some second block, which is both the second block to appear from left to right, and the second block to appear from right to left.
The probability of this event determined by $1 < j < k < n$ is the probability that the set $[j] \cup ([n] \setminus [k-1])$ of $j + 1 + n - k $ elements is split by the partition into the two particular subsets 
$[j-1] \cup ([n] \setminus [k])$ and $\{j\} \cup \{k\}$ of $j-1 + n - k $ and $2$ respectively. 
Hence the $p( j-1 + n - k , 2 )$, by the exchangeability and consistency  properties of the random partitions of various subsets of $[n]$.  
\item $\sum_{j=1}^{n-1} p(j,n-j)$ is the probability of the event that there are exactly two blocks $[j]$ and $[n] \setminus [j]$ for some $1 \le j < n$.
\item $\sum_{1 < j < k < n } p(j, n-k+1)$ is the sum of the probabilities that there are two or more blocks, with $1$ and $n$ in different blocks, with $j$ the
first element of the block containing $n$, which is the second block to appear from left to right, and $k$ the last element of the block containing $1$, which is the second block to appear from right to left.
The probability of this event determined by $1 < j < k < n$ is the probability that the set $[j] \cup ([n] \setminus [k-1])$ of $j + 1 + n - k $ elements is split by the partition into the two particular subsets 
$[j-1] \cup \{k\}$ and $\{j \} \cup ([n] - [k])$ of sizes $j$ and $n-k+1$ respectively. Hence the $p(j,n-k+1)$, again by the exchangeability and consistency  properties of the random partitions of various subsets of $[n]$. 
\end{itemize}
\end{proof}

\quad In this classification of five kinds of terms contributing to the probability $u_{2:n}$, the first three kinds account for all cases in which $1$ and $n$ fall in the same block, while
the last two kinds account for all cases in which $1$ and $n$ fall in different blocks.  The double sum for the third kind of term is $0$ unless $n \ge 4$, in which case it always simplifies to a
single sum by grouping terms according to the value $h$ of $j-1 + n - k$ :
\begin{align}
\label{prob1} \sum_{1 < j < k < n } p( j-1 + n - k , 2 )  &= \sum_{h=2}^ {n-2} (h-1) p( h,2 )  \\
\label{prob1th}  & \eqth \left[ p_\theta(2) - p_\theta(n) - (n-2) p_\theta(n-1,1) \right] \, p_\theta (2)  
\end{align}
with further simplification as indicated by $\eqth$ for the Ewens $(\theta)$ model.
For the Ewens $(\theta)$ model, the $p(h,2)$ in \eqref{prob1} becomes
$$
p_\theta( h,2 )  = \frac{ \theta (1)_{h-1} }{ (1 + \theta )_{h+1} }
$$
while the expression in \eqref{prob1th} features
$$
p_\theta(2) = \frac{ 1 } { 1 + \theta };
\qquad
p_\theta(n) = \frac{ (1)_{n-1} } { (1 + \theta )_{n-1}} ;
\qquad
p_\theta(n-1,1) = \frac{ \theta \, (1)_{n-2} } { (1 + \theta )_{n-1} }.
$$
The evaluation $\eqth$ in \eqref{prob1th} is easily checked algebraically, by first checking it for $n = 4$, then checking the equality of differences
as $n$ is incremented. This evaluation \eqref{prob1th} is an expression of the well known characteristic property of {\em non-interference} in the Ewens $(\theta)$ model, according to which, given that $1$ and $n$ fall in the
same block of some size $b$ with $2 \le b < n$, the remaining $n-b$ elements are partitioned by according to the Ewens $(\theta)$ model for $n-b$ elements.
The sums in \eqref{prob1} evaluate the probability that $1$ and $n$ fall in the same block, whose size $b$ is at most $n-2$, and that the partition of the
remaining $n-b \ge 2$ elements puts the least of these element in the same block as the greatest of these elements.  For a general EPPF the
probability that $1$ and $n$ fall in the same block, whose size $b$ is at most $n-2$, is $p(2) - p(n) - (n-2) p(n-1,1)$, as in the first factor of 
\eqref{prob1th} for $p = p_\theta$.  For the Ewens $(\theta)$ model, given this event and the size $b \le n-2$ of the block containing $1$ and $n$, the probability that the remaining $n-b \ge 2$ elements 
have their least and greatest elements in the same block is just $p_\theta(2)$, regardless of the value of $b$. Hence the factorization in the expression of \eqref{prob1th} for the Ewens $(\theta)$ model.

\quad While the sum of the first three kinds of terms in \eqref{5kinds} can be simplified as above in the Ewens $(\theta)$ model, even for $\theta = 1$ there is no comparable simplification for the
sum of the last two kinds of terms in \eqref{5kinds}, representing the probability of the event that $1$ and $n$ fall in different blocks,  
while the same union of the first two blocks is found by examining elements from left to right as in examining elements from right to left.
Asymptotics as $n \to \infty$ are easy for the sum of the first three kinds of terms in \eqref{5kinds}. The limit of the contribution of these three terms is $p_\theta(2)^2 = (1 + \theta)^{-2} \eqone 1/4$.
As for the remaining two kinds of terms, it is obvious that $\sum_{j=1}^{n-1} p(j,n-j) \to 0$ for any partition structure, since this is the probability that there are only two classes, and
all elements of one class appear in a sample of size $n$ before all members of the other class. So for the Ewens $(\theta)$ model this gives
\begin{align}
\lim_{n \to \infty} u_{2:n} &\overset{\theta}{=} (1 + \theta)^{-2} + \lim_{n\to \infty} \sum_{1 < j < k < n } \frac{ \theta (1)_{j-1}  (1 )_{n-k} }{ ( 1 + \theta )_{j + n - k } } \notag \\
&= (1 + \theta)^{-2} + \sum_{j=2}^{\infty} \frac{\theta(1)_{j-1}}{(\theta+j-1)(\theta+1)_{j}}  
\end{align}
where the limit can also be evaluated as an integral with respect to the joint distribution of $P_1 = W_1$ and $P_2 = (1-W_1)W_2$ for $W_i$ independent beta $(1,\theta)$ variables.
We can write
\begin{equation} 
u_2 \overset{\theta}{=}  \frac{1}{(\theta+1)^2} + \frac{1}{\theta+1}\left (\thrftwo{1}{1}{\theta}{\theta+1}{\theta+2}{1}-1\right ),
\end{equation}
in terms of the generalized hypergeometric function $_3F_2$.
Lima \cite[Lemma 1]{lima2012rapidly} 
gives the following formula for {\em Catalan's constant} $G$:
\begin{equation}\label{eq:lima}
\frac{1}{2}\thrftwo{1}{1}{\frac{1}{2}}{\frac{3}{2}}{\frac{3}{2}}{1} = G = \beta(2),
\end{equation}
where $\beta(s) := \sum_{n=1}^{\infty}\frac{(-1)^n}{(2n+1)^s}$ for $s>0$.
By manipulating the hypergeometric $_3 F_2$ function, one can see that for $\theta = n+1/2$ where $n$ is an integer, $u_2$ is of the form $q+rG$, where $q$ and $r$ are rational numbers.
Lima's articles
\cite{MR3357692, MR2968884} 
contain many related formulas, and references to zeta and beta values.

\bigskip
{\bf Acknowledgement:} 
We thank David Aldous for various pointers to the literature.

\bibliographystyle{plain}
\bibliography{unique}

\begin{thebibliography}{10}

\bibitem{AN}
Sergei Abramovich and Yakov~Yu. Nikitin.
\newblock On the {P}robability of {C}o-primality of two {N}atural {N}umbers
  {C}hosen at {R}andom: {F}rom {E}uler identity to {H}aar {M}easure on the
  {R}ing of {A}deles.
\newblock {\em Bernoulli News}, 24(1):7--13, 2017.

\bibitem{abramowitz1964handbook}
Milton Abramowitz and Irene~A Stegun.
\newblock {\em Handbook of mathematical functions: with formulas, graphs, and
  mathematical tables}, volume~55.
\newblock Courier Corporation, 1964.

\bibitem{AET}
Shigeki Akiyama, Shigeki Egami, and Yoshio Tanigawa.
\newblock Analytic continuation of multiple zeta-functions and their values at
  non-positive integers.
\newblock {\em Acta Arith.}, 98(2):107--116, 2001.

\bibitem{AldouS}
David Aldous and J.~Michael Steele.
\newblock The objective method: probabilistic combinatorial optimization and
  local weak convergence.
\newblock In {\em Probability on discrete structures}, volume 110 of {\em
  Encyclopaedia Math. Sci.}, pages 1--72. Springer, Berlin, 2004.

\bibitem{ABR}
Kenneth~S. Alexander, Kenneth Baclawski, and Gian-Carlo Rota.
\newblock A stochastic interpretation of the {R}iemann zeta function.
\newblock {\em Proc. Nat. Acad. Sci. U.S.A.}, 90(2):697--699, 1993.

\bibitem{AKO}
Takashi Aoki, Yasuhiro Kombu, and Yasuo Ohno.
\newblock A generating function for sums of multiple zeta values and its
  applications.
\newblock {\em Proc. Amer. Math. Soc.}, 136(2):387--395, 2008.

\bibitem{AK}
Tsuneo Arakawa and Masanobu Kaneko.
\newblock Multiple zeta values, poly-{B}ernoulli numbers, and related zeta
  functions.
\newblock {\em Nagoya Math. J.}, 153:189--209, 1999.

\bibitem{arguin_maxima_2017}
Louis-Pierre Arguin, David Belius, and Adam~J. Harper.
\newblock Maxima of a randomized {{Riemann}} zeta function, and branching
  random walks.
\newblock {\em The Annals of Applied Probability}, 27(1):178--215, February
  2017.

\bibitem{ABN}
B.~C. Arnold, N.~Balakrishnan, and H.~N. Nagaraja.
\newblock {\em Records}.
\newblock Wiley Series in Probability and Statistics: Probability and
  Statistics. John Wiley \& Sons, Inc., New York, 1998.
\newblock A Wiley-Interscience Publication.

\bibitem{abtlogcs}
Richard Arratia, A.~D. Barbour, and Simon Tavar\'e.
\newblock {\em Logarithmic combinatorial structures: a probabilistic approach}.
\newblock EMS Monographs in Mathematics. European Mathematical Society (EMS),
  Z\"urich, 2003.

\bibitem{bailey_experimental_1994}
David~H. Bailey, Jonathan~M. Borwein, and Roland Girgensohn.
\newblock Experimental evaluation of {{Euler}} sums.
\newblock {\em Experimental Mathematics}, 3(1):17--30, 1994.

\bibitem{Berry}
M~Berry.
\newblock Riemann's {Z}eta function: {A} model for quantum chaos?
\newblock In {\em Quantum chaos and statistical nuclear physics}, pages 1--17.
  Springer, 1986.

\bibitem{Beukers}
F.~Beukers.
\newblock A note on the irrationality of {$\zeta (2)$}\ and {$\zeta (3)$}.
\newblock {\em Bull. London Math. Soc.}, 11(3):268--272, 1979.

\bibitem{BPY}
Philippe Biane, Jim Pitman, and Marc Yor.
\newblock Probability laws related to the {J}acobi theta and {R}iemann zeta
  functions, and {B}rownian excursions.
\newblock {\em Bull. Amer. Math. Soc. (N.S.)}, 38(4):435--465, 2001.

\bibitem{Bombieri}
Enrico Bombieri.
\newblock The {R}iemann hypothesis.
\newblock In {\em The millennium prize problems}, pages 107--124. Clay Math.
  Inst., Cambridge, MA, 2006.

\bibitem{borwein_intriguing_1995}
David Borwein and Jonathan~M. Borwein.
\newblock On an {{Intriguing Integral}} and {{Some Series Related}} to
  $zeta(4)$.
\newblock {\em Proceedings of the American Mathematical Society}, 123(4):1191,
  April 1995.

\bibitem{BBG}
David Borwein, Jonathan~M. Borwein, and Roland Girgensohn.
\newblock Explicit evaluation of {E}uler sums.
\newblock {\em Proc. Edinburgh Math. Soc. (2)}, 38(2):277--294, 1995.

\bibitem{BBBL}
Jonathan~M. Borwein, David~M. Bradley, David~J. Broadhurst, and Petr Lison{\v
  e}k.
\newblock Combinatorial aspects of multiple zeta values.
\newblock {\em Electron. J. Combin.}, 5:Research Paper 38, 12, 1998.

\bibitem{Olivier}
Olivier Bouillot.
\newblock On {H}urwitz multizeta functions.
\newblock {\em Adv. in Appl. Math.}, 71:68--124, 2015.

\bibitem{GF17}
Jos{\'e}~Ignacio Burgos~Gil and Javier Fres{\'a}n.
\newblock Multiple zeta values: from numbers to motives.
\newblock Available at {\url{http://javier.fresan.perso.math.cnrs.fr/mzv.pdf}}.

\bibitem{chang_ladder_1997}
Joseph~T. Chang and Yuval Peres.
\newblock Ladder heights, {{Gaussian}} random walks and the {{Riemann}} zeta
  function.
\newblock {\em The Annals of Probability}, pages 787--802, 1997.

\bibitem{Comtet74}
Louis Comtet.
\newblock {\em Advanced combinatorics}.
\newblock D. Reidel Publishing Co., Dordrecht, enlarged edition, 1974.
\newblock The art of finite and infinite expansions.

\bibitem{DG}
Peter Donnelly and Geoffrey Grimmett.
\newblock On the asymptotic distribution of large prime factors.
\newblock {\em J. London Math. Soc. (2)}, 47(3):395--404, 1993.

\bibitem{Edwards}
H.~M. Edwards.
\newblock {\em Riemann's zeta function}.
\newblock Dover Publications, Inc., Mineola, NY, 2001.
\newblock Reprint of the 1974 original [Academic Press, New York; MR0466039 (57
  \#5922)].

\bibitem{Ev}
Torsten Ekedahl and Gerard van~der Geer.
\newblock Cycles representing the top {C}hern class of the {H}odge bundle on
  the moduli space of abelian varieties.
\newblock {\em Duke Math. J.}, 129(1):187--199, 2005.

\bibitem{ERS}
P.~Erdős, A.~R\'enyi, and P.~Sz\"usz.
\newblock On {E}ngel's and {S}ylvester's series.
\newblock {\em Ann. Univ. Sci. Budapest. E\"otv\"os. Sect. Math.}, 1:7--32,
  1958.

\bibitem{EL}
C.~J.~A. Evelyn and E.~H. Linfoot.
\newblock On a problem in the additive theory of numbers.
\newblock {\em Ann. Math.}, 32:261--270, 1931.

\bibitem{Feller}
William Feller.
\newblock {\em An introduction to probability theory and its applications.
  {V}ol. {I}}.
\newblock Third edition. John Wiley \& Sons Inc., New York, 1968.

\bibitem{Frieze}
A.~M. Frieze.
\newblock On the value of a random minimum spanning tree problem.
\newblock {\em Discrete Appl. Math.}, 10(1):47--56, 1985.

\bibitem{Gsmall}
Alexander Gnedin, Alex Iksanov, and Uwe Roesler.
\newblock Small parts in the {B}ernoulli sieve.
\newblock In {\em Fifth {C}olloquium on {M}athematics and {C}omputer
  {S}cience}, Discrete Math. Theor. Comput. Sci. Proc., AI, pages 235--242.
  Assoc. Discrete Math. Theor. Comput. Sci., Nancy, 2008.

\bibitem{GIM}
Alexander Gnedin, Alexander Iksanov, and Alexander Marynych.
\newblock The {B}ernoulli sieve: an overview.
\newblock In {\em 21st {I}nternational {M}eeting on {P}robabilistic,
  {C}ombinatorial, and {A}symptotic {M}ethods in the {A}nalysis of {A}lgorithms
  ({A}of{A}'10)}, Discrete Math. Theor. Comput. Sci. Proc., AM, pages 329--341.
  Assoc. Discrete Math. Theor. Comput. Sci., Nancy, 2010.

\bibitem{Gnedinsieve}
Alexander~V. Gnedin.
\newblock The {B}ernoulli sieve.
\newblock {\em Bernoulli}, 10(1):79--96, 2004.

\bibitem{GINR}
Alexander~V. Gnedin, Alexander~M. Iksanov, Pavlo Negadajlov, and Uwe R\"osler.
\newblock The {B}ernoulli sieve revisited.
\newblock {\em Ann. Appl. Probab.}, 19(4):1634--1655, 2009.

\bibitem{Gut}
Allan Gut.
\newblock Some remarks on the {R}iemann zeta distribution.
\newblock {\em Rev. Roumaine Math. Pures Appl.}, 51(2):205--217, 2006.

\bibitem{Hoffman}
Michael~E. Hoffman.
\newblock Multiple harmonic series.
\newblock {\em Pacific J. Math.}, 152(2):275--290, 1992.

\bibitem{Hoffman05}
Michael~E. Hoffman.
\newblock Algebraic aspects of multiple zeta values.
\newblock In {\em Zeta functions, topology and quantum physics}, volume~14 of
  {\em Dev. Math.}, pages 51--73. Springer, New York, 2005.

\bibitem{Kaluza}
Th. Kaluza.
\newblock \"{U}ber die {K}oeffizienten reziproker {P}otenzreihen.
\newblock {\em Math. Z.}, 28(1):161--170, 1928.

\bibitem{KSarnak}
Nicholas~M. Katz and Peter Sarnak.
\newblock Zeroes of zeta functions and symmetry.
\newblock {\em Bull. Amer. Math. Soc. (N.S.)}, 36(1):1--26, 1999.

\bibitem{Kirsten}
Klaus Kirsten.
\newblock Basic zeta functions and some applications in physics.
\newblock In {\em A window into zeta and modular physics}, volume~57 of {\em
  Math. Sci. Res. Inst. Publ.}, pages 101--143. Cambridge Univ. Press, 2010.

\bibitem{Li}
Xian-Jin Li.
\newblock The positivity of a sequence of numbers and the {R}iemann hypothesis.
\newblock {\em J. Number Theory}, 65(2):325--333, 1997.

\bibitem{MR2968884}
F.~M.~S. Lima.
\newblock An {E}uler-type formula for {$\beta(2n)$} and closed-form expressions
  for a class of zeta series.
\newblock {\em Integral Transforms Spec. Funct.}, 23(9):649--657, 2012.

\bibitem{MR3357692}
F.~M.~S. Lima.
\newblock A simpler proof of a {K}atsurada's theorem and rapidly converging
  series for {$\zeta(2n+1)$} and {$\beta(2n)$}.
\newblock {\em Ann. Mat. Pura Appl. (4)}, 194(4):1015--1024, 2015.

\bibitem{lima2012rapidly}
FMS Lima.
\newblock A rapidly converging ramanujan-type series for catalan's constant.
\newblock {\em arXiv preprint arXiv:1207.3139}, 2012.

\bibitem{LH}
Gwo~Dong Lin and Chin-Yuan Hu.
\newblock The {R}iemann zeta distribution.
\newblock {\em Bernoulli}, 7(5):817--828, 2001.

\bibitem{Montgomery}
H.~L. Montgomery.
\newblock The pair correlation of zeros of the zeta function.
\newblock In {\em Analytic number theory ({P}roc. {S}ympos. {P}ure {M}ath.,
  {V}ol. {XXIV}, {S}t. {L}ouis {U}niv., {S}t. {L}ouis, {M}o., 1972)}, pages
  181--193. Amer. Math. Soc., Providence, R.I., 1973.

\bibitem{Nevzorov}
V.~B. Nevzorov.
\newblock {\em Records: mathematical theory}.
\newblock Translations of Mathematical Monographs, vol. 194. American
  Mathematical Society, Providence, RI, 2001.

\bibitem{Odly}
A.~M. Odlyzko.
\newblock The {$10^{22}$}-nd zero of the {R}iemann zeta function.
\newblock In {\em Dynamical, spectral, and arithmetic zeta functions ({S}an
  {A}ntonio, {TX}, 1999)}, volume 290 of {\em Contemp. Math.}, pages 139--144.
  Amer. Math. Soc., Providence, RI, 2001.

\bibitem{Ohno}
Yasuo Ohno.
\newblock A generalization of the duality and sum formulas on the multiple zeta
  values.
\newblock {\em J. Number Theory}, 74(1):39--43, 1999.

\bibitem{P17}
Jim Pitman.
\newblock Extremes and gaps in sampling from a residual allocation model.
\newblock In preparation.

\bibitem{pitmanbook}
Jim Pitman.
\newblock {\em Combinatorial stochastic processes}, volume 1875 of {\em Lecture
  Notes in Mathematics}.
\newblock Springer-Verlag, Berlin, 2006.

\bibitem{PT17}
Jim Pitman and Wenpin Tang.
\newblock Regenerative random permutations of integers.
\newblock {\em arXiv:1704.01166}, 2017.

\bibitem{PY17}
Jim Pitman and Yuri Yakubovich.
\newblock Extremes and gaps in sampling from a {G}{E}{M} random discrete
  distribution.
\newblock {\em arXiv:1701.06294}, 2017.

\bibitem{renyi}
A.~R\'enyi.
\newblock A new approach to the theory of {E}ngel's series.
\newblock {\em Ann. Univ. Sci. Budapest. E\"otv\"os Sect. Math.}, 5:25--32,
  1962.

\bibitem{Vervaat}
Wim Vervaat.
\newblock Limit theorems for records from discrete distributions.
\newblock {\em Stochastic Processes Appl.}, 1:317--334, 1973.

\bibitem{Williams}
David Williams.
\newblock Brownian motion and the {R}iemann zeta-function.
\newblock In {\em Disorder in physical systems}, Oxford Sci. Publ., pages
  361--372. Oxford Univ. Press, New York, 1990.

\bibitem{Witten}
Edward Witten.
\newblock On quantum gauge theories in two dimensions.
\newblock {\em Comm. Math. Phys.}, 141(1):153--209, 1991.

\bibitem{Zagier}
Don Zagier.
\newblock Values of zeta functions and their applications.
\newblock In {\em First {E}uropean {C}ongress of {M}athematics, {V}ol.\ {II}
  ({P}aris, 1992)}, volume 120 of {\em Progr. Math.}, pages 497--512.
  Birkh\"auser, Basel, 1994.

\bibitem{Zhao}
Jianqiang Zhao.
\newblock Analytic continuation of multiple zeta functions.
\newblock {\em Proc. Amer. Math. Soc.}, 128(5):1275--1283, 2000.

\end{thebibliography}
\end{document}